\documentclass[11pt,a4paper]{article}

\usepackage{geometry}                
\usepackage[parfill]{parskip}    

\usepackage[utf8]{inputenc}

\usepackage{graphicx}

\usepackage{amsmath, amssymb, amsthm, amsfonts}
\usepackage{mathtools}
\usepackage{bbm} 

\usepackage{hyperref}

\usepackage[svgnames]{xcolor}
\usepackage[breakable,theorems]{tcolorbox}
\usepackage[font=small,labelfont=bf]{caption}

\usepackage{libertine}
\usepackage[libertine]{newtxmath}
\usepackage{inconsolata}
\usepackage{microtype}

\DeclareMathAlphabet{\mathcal}{OMS}{cmsy}{m}{n}

\usepackage{enumitem}

\usepackage{booktabs}

\usepackage{csquotes}

\usepackage{cleveref}

\usepackage{pgf}

\usepackage[backend=biber,
           giveninits=true,
           mincitenames=1,
           maxcitenames=2,
           maxbibnames=99,
           uniquename=false,
           uniquelist=false,
           style=ext-authoryear,
           url = false,
           eprint = false,
           isbn = false,
           doi = false,
           articlein = false,
           dashed = false,
          ]{biblatex}

\numberwithin{equation}{section}

\newtcbtheorem[number within=section, Crefname={Definition}{Definition}]%
  {definition}{Definition}{theorem style=plain, 
                           coltitle=green!30!black,
                           colframe=white,
                           colback=white,
                           breakable}{def}

\newtcbtheorem[Crefname={Assumption}{Assumption}]
  {assumption}{Assumption}{theorem style=plain, 
                           description delimiters none,
                           coltitle=green!30!black,
                           colframe=white,
                           colback=white,
                           breakable}{def}

\newtcbtheorem[use counter from=definition, number within=section, Crefname={Remark}{Remark}]%
  {remark}{Remark}{theorem style=plain, 
                   coltitle=red!40!black,
                   colframe=white,
                   colback=white,
                   breakable}{rem}

\newtcbtheorem[use counter from=definition, number within=section, Crefname={Lemma}{Lemma}]%
  {lemma}{Lemma}{theorem style=plain, 
                 coltitle=orange!45!black,
                 colframe=white,
                 colback=white,
                 breakable}{lem}

\newtcbtheorem[use counter from=definition, number within=section, Crefname={Theorem}{Theorem}]%
  {theorem}{Theorem}{theorem style=plain, 
                     coltitle=orange!45!black,
                     colframe=white,
                     colback=white,
                     breakable}{thm}

\newtcbtheorem[use counter from=definition, number within=section, Crefname={Corollary}{Corollary}]%
  {corollary}{Corollary}{theorem style=plain, 
                         coltitle=orange!45!black,
                         colframe=white,
                         colback=white,
                         breakable}{cor}

\newtcbtheorem[use counter from=definition, number within=section, Crefname={Proposition}{Proposition}]%
  {proposition}{Proposition}{theorem style=plain, 
                             coltitle=orange!45!black,
                             colframe=white,
                             colback=white,
                             breakable}{prop}

\newtcbtheorem[use counter from=definition, number within=section, Crefname={Example}{Example}]%
  {example}{Example}{theorem style=plain, 
                     coltitle=white!20!black,
                     colframe=white,
                     colback=white,
                     breakable}{ex}

\newcommand{\X}{\mathcal{X}}
\newcommand{\Y}{\mathcal{Y}}
\newcommand{\Z}{\mathcal{Z}}
\newcommand{\U}{\mathcal{U}}
\renewcommand{\P}{\mathcal{P}}
\newcommand{\C}{\mathcal{C}}
\newcommand{\F}{\mathcal{F}}

\newcommand{\B}{\mathcal{B}}
\newcommand{\G}{\mathcal{G}}

\newcommand{\XY}{\X\times\Y}
\newcommand{\XC}{\mathcal{X}}
\newcommand{\YC}{\mathcal{Y}}

\newcommand{\PC}{\mathcal{P}}
\newcommand{\CC}{\mathcal{C}}
\newcommand{\DC}{\mathcal{D}}

\newcommand{\FC}{\mathcal{F}}

\newcommand{\RR}{\mathbb{R}}
\newcommand{\NN}{\mathbb{N}}
\newcommand{\EE}{\mathbb{E}}
\newcommand{\Lip}{\mathrm{Lip}}

\newcommand{\Exp}{\mathbb{E}}
\newcommand{\Var}{\mathrm{Var}}

\newcommand{\Unif}{\mathrm{Unif}}
\newcommand{\Bin}{\mathrm{Bin}}
\newcommand{\Mult}{\mathrm{Mult}}

\newcommand{\diam}{\mathrm{diam}}

\newcommand{\mint}[1]{{#1\hspace{0.05em}}}

\newcommand{\cX}{c_{\hspace{-0.1em}\X}}
\newcommand{\cY}{c_{\Y}}
\newcommand{\cXY}{\cX\oplus\cY}

\newcommand{\dif}{\,\mathrm{d}}
\newcommand{\supp}{\mathrm{supp}}

\newcommand{\range}{\mathrm{range}}
\newcommand{\norm}[1]{\left\|#1\right\|}

\newcommand{\T}{T}

\newcommand{\BC}{\mathrm{BC}} 
\newcommand{\DBC}{\mathrm{DBC}} 

\newcommand{\DIM}{\mathrm{DIM}} 
\newcommand{\LIP}{\mathrm{L}} 
\newcommand{\SMOOTH}{\mathrm{S}} 

\newcommand{\covering}{\mathcal{N}}

\newcommand{\RRplus}{\RR_{+}}

\newcommand{\frar}[2]{{#1/#2}}
\newcommand{\frab}[2]{{(#1)/#2}}

\bibliography{sources}

\begin{document}

\title{Empirical Optimal Transport: Rates of Convergence in Unbounded Settings}
\title{Convergence of Empirical Optimal Transport\\ in Unbounded Settings}
\author{Thomas Staudt and Shayan Hundrieser}
\date{\today}

\maketitle

\begin{abstract}
   \noindent
   In compact settings, the convergence rate of the empirical optimal transport cost to its population value is well understood for a wide class of spaces and cost functions.
   In unbounded settings, however, hitherto available results require strong assumptions on the ground costs and the concentration of the involved measures.
   In this work, we pursue a decomposition-based approach to generalize the convergence rates found in compact spaces to unbounded settings under generic moment assumptions that are sharp up to an arbitrarily small $\epsilon > 0$.
   Hallmark properties of empirical optimal transport on compact spaces, like the recently established adaptation to lower complexity, are shown to carry over to the unbounded case.
\end{abstract}

{\small

}

\section{Introduction}
\label{sec:introduction}

Optimal transportation refers to the problem of transforming a given probability distribution into another one while minimizing the average costs of movement.
The \emph{optimal transport cost} between two probability measures $\mu\in\P(\X)$ and $\nu\in\P(\Y)$ on Polish spaces $\X$ and $\Y$ is defined~as
\begin{equation}\label{eq:ot}
  \T_c(\mu, \nu)
  =
  \inf_{\pi\in\CC(\mu, \nu)} \mint{\pi}c
  =
  \inf_{\pi\in\CC(\mu, \nu)} \int c(x, y)\dif\pi(x, y).
\end{equation}
The cost function $c\colon\XY \to\RRplus = [0,\infty)$ expresses the effort necessary to move one unit of mass from $x\in\X$ to $y\in\Y$.
The set $\CC(\mu, \nu)$ contains all measures $\pi\in\P(\XY)$ with marginals $\mu$ and $\nu$, meaning that $\pi(A\times\Y) = \mu(A)$ and $\pi(\X\times B) = \nu(B)$ for all Borel sets $A\subset\X$ and $B\subset\Y$.
Elements of this set are called transport plans.
Minimizers $\pi$ of \eqref{eq:ot} are named \emph{optimal transport plans} and exist under mild conditions.

Early work on optimal transportation dates back to \textcite{monge}, before \textcite{kantorovich} established its modern formulation.
In the past few decades, insights into optimal transport theory have been advanced substantially, and a rich mathematical landscape has emerged;
see \textcite{villani2008optimal}, \textcite{santambrogio2015optimal}, \textcite{peyre2019computational}, \textcite{panaretos2020invitation}, and \textcite{ambrosio2021lectures} for monographs that engage with the topic from analytical, geometrical, statistical, and computational perspectives.
At the same time, optimal transport based methodologies have found applications in a wide variety of fields, ranging from biology \parencite{Schiebinger19,tameling2021colocalization} and economics \parencite{galichon2016optimal} to computer vision and machine learning \parencite{gulrajani2017improved, kolkin2019style}.
Despite the staggering amount of research, however, several obstacles still affect its practical utility.
Important concerns are the efficient computation, approximation, and estimation of optimal transport related quantities on large datasets and in high dimensions.

The key objective of this manuscript is to present a generic decomposition approach for the analysis of empirically estimated optimal transport in unbounded settings.
For $n,m\in\NN$, let $(X_i)_{i=1}^n \sim \mu^{\otimes n}$ and $(Y_j)_{j=1}^m\sim \nu^{\otimes m}$ be i.i.d.\ (independent and identically distributed) samples that are also independent of one another.
The corresponding empirical measures are defined~by
\begin{equation*}
  \hat\mu_n \coloneqq  \frac{1}{n}\sum_{i=1}^n \delta_{X_i}
  \qquad\text{and}\qquad
  \hat\nu_m \coloneqq  \frac{1}{m}\sum_{j=1}^m \delta_{Y_j},
\end{equation*}
where $\delta_z$ denotes the point measure at location $z$.
We are mainly interested in the convergence of the \emph{empirical optimal transport cost} $\T_c(\hat\mu_n, \hat\nu_m)$ to its population counterpart $\T_c(\mu, \nu)$ as $n$ and $m$ tend to infinity.

The convergence of $\T_c(\hat\mu_n, \hat\nu_m)$ towards $\T_c(\mu, \nu)$ has been a prime subject of interest in a number of recent publications (see \Cref{sec:related} for a detailed overview).
The picture is most complete in compact settings, where a fairly generic theory covers arbitrary measures $\mu$ and $\nu$ and a wide range of costs and spaces \parencite{boissard2014mean,weed2019sharp,hundrieser2022empirical}.
In unbounded settings, on the contrary, available statements are more scattered and either require specific costs, only apply when $\mu = \nu$ \parencite{dereich2013constructive,fournier2015rate}, or expect strong concentration properties from the involved measures $\mu$ and $\nu$ \parencite{manole2021sharp}.
One main reason for this discrepancy is that the function classes appearing in the dual formulation of the optimal transport problem, which is the principal tool to bound $\T_c(\hat\mu_n, \hat\nu_m) - \T_c(\mu, \nu)$ via empirical process theory, are much harder to control in unbounded scenarios.
Achieving sharp convergence rates for unbounded domains, however, is crucial for a better understanding of the stability of optimal transport costs under plug-in estimators.

This work closes the gap between the results available in compact and non-compact settings.
The underlying idea is to decompose an optimal transport plan between $\mu$ and $\nu$ into a convex combination of sub-plans with compactly supported marginals, reducing the analysis of $\T_c(\hat\mu_n, \hat\nu_n) - \T_c(\mu, \nu)$ to compact sub-problems.
The error that results from miss-alignments of the empirical data, which does not follow the decomposition perfectly, is controlled via a composition bound (\Cref{lem:composition-bound}).
We pursue this strategy for generic lower semi-continuous cost functions that are marginally bounded in the sense
\begin{equation}\label{eq:marginal-cost-bound}
  c(x, y)
  \le
  \cX(x) + \cY(y)
  \qquad
  \text{for all $(x,y)\in\XY$},
\end{equation}
where $\cX\colon \X\to\RRplus$ and $\cY\colon \Y\to\RRplus$ are lower semi-continuous functions.
We also abbreviate this condition by writing $c \le \cX\oplus\cY$.
To formulate our result, we define the $\cX$- and $\cY$-balls $B_\X(r) = \cX^{-1}[0, r]$ and $B_\Y(r) = \cY^{-1}[0, r]$ with radius $r \ge 0$ and place the following assumption on the convergence of the empirical optimal transport cost on bounded set.
We denote the minimum (or maximum) of two real numbers $a,b\in\RR$ by $a \wedge b$ (or $a \vee b$).

\begin{assumption*}{$\BC(\kappa, \alpha)$}{}
  Let $c$ be of form \eqref{eq:marginal-cost-bound} and let there be $\kappa
  > 0$ and $\alpha\in(0, 1/2]$ such that the \emph{bounded convergence
  assumption}
  \begin{equation}\label{eq:bounded-convergence}
    \Exp\,\big|T_c(\hat\mu_{n}, \hat\nu_{m}) - T_c(\mu, \nu)\big|
    \le
    \kappa\,r\,(n\wedge m)^{-\alpha}
  \end{equation}
  holds for all $r \ge 1$, $\mu\in\P\big(B_\X(r)\big)$, $\nu\in\P\big(B_\Y(r)\big)$, and
  $n, m\in\NN$. 
\end{assumption*}

Note that the required linear scaling in $r$ is usually the correct one (see \Cref{sec:examples}).
For example, in Euclidean settings with costs $c(x,y) = \|x-y\|^p$ for $p>0$, suitable marginal bounds are $\cX(x) = \cY(x) = 2^p\|x\|^p$.
Then, the right hand side of \eqref{eq:bounded-convergence} grows with the $p$-th power of the Euclidean diameter of the supports of $\mu$ and $\nu$.

In \Cref{sec:proof}, we prove the following result via the described decomposition strategy.

\begin{theorem}{unbounded rates}{unbounded-rates}
  Let $\X$ and $\Y$ be Polish, $\mu\in\P(\X)$, $\nu\in\P(\Y)$, and $c\colon\XY\to\RRplus$ lower semi-continuous such that $\BC(\kappa, \alpha)$ is satisfied.
  Assume there are $1 < s \le 2$ and $\epsilon > 0$ such that the moments $\mint\mu \cX^{s+\epsilon}$ and $\mint\nu \cY^{s+\epsilon}$ are finite.
  Then, for all $n,m\in\NN$,
  \begin{equation}\label{eq:unbounded-rates}
    \Exp\,\big|T_c(\hat\mu_{n}, \hat\nu_{m}) - T_c(\mu, \nu)\big|
    \lesssim
    (n\wedge m)^{-\alpha} + (n\wedge m)^{-\frab{s-1}{s}}.
  \end{equation}
\end{theorem}

This result identifies two distinct statistical regimes.
The first one ($\alpha \le \frab{s-1}{s}$) occurs if sufficiently many moments exist.
Then, optimal transport between bounded sets is statistically the dominating operation, and the rates for bounded settings determine the unbounded ones.
The second regime ($ \alpha > \frab{s - 1}{s}$) occurs if transport between (or into) the tail of the distributions dominates.
In this case, the exponent $\frab{s - 1}{s}$ for $1 < s \le 2$ corresponds to the expected rate for the empirical mean of a random variable that has moments up to order $s$ (see \Cref{subsec:lowerbound:tail}).
The following comments provide additional context regarding \Cref{thm:unbounded-rates}.

\begin{enumerate}[leftmargin=2em, label={\arabic*)}]

  \item \label{it:log}
    For simplicity of exposition, Assumption $\BC(\kappa, \alpha)$ does not cover rates with additional logarithmic factors, like $n^{-1/2}\log(n+1)$.
    These rates commonly occur in upper bounds for compact settings, e.g., in dimension $2$ for Lipschitz costs or dimension $4$ for smooth costs (see \Cref{sec:examples}).
    A look at the proof reveals that these $\log$ factors can be included into \Cref{thm:unbounded-rates} and affect the resulting rate as expected (see \Cref{subsec:composition}).

  \item \label{it:one-sample} Rates for the one-sample quantity $\T_c(\hat\mu_n, \nu)$ can be derived from \eqref{eq:unbounded-rates} via the inequality
    \begin{equation*}
      \Exp\,\big|\T_c(\hat\mu_{n}, \nu) - \T_c(\mu, \nu)\big|
      \le
      \Exp\,\big|\T_c(\hat\mu_{n}, \hat\nu_{m}) - \T_c(\mu, \nu)\big|,
    \end{equation*}
    which holds for any $m\in\NN$ (see \Cref{subsec:composition}).
    In some configurations, like specific semi-discrete settings, where we sample from a finitely supported measure $\mu$, the moment requirement $\mint\nu\cY^{s+\epsilon} < \infty$ can be relaxed in this case (\cite{del2022central}).
    However, in general, this condition on $\nu$ might actually be necessary even when only sampling from $\mu$ (see \Cref{sec:discussion} for a discussion).
    Analogous results hold for $T_c(\mu, \hat\nu_m)$.

  \item \label{it:mu-equals-nu}
    If $\mu = \nu$ and the costs satisfy $c(x,x) = 0$ and $\cX = \cY$, the plan $\pi$ induced by the identity function is an optimal transport plan, and the restricted marginals that are constructed in the proof of \Cref{thm:unbounded-rates} will coincide.
    In this case, Assumption~$\BC(\kappa, \alpha)$ only has to be satisfied for equal measures $\mu = \nu$ as well.
    This can be of interest in settings where improved convergence rates for, say, $\T_c(\hat\mu_n, \mu)$, are known in the bounded setting.

  \item \label{it:known-pi}
    Generalizing the previous point, it suffices to show Assumption~$\BC(\kappa, \alpha)$ for measures $\mu\in\P\big(B_\X(r)\big)$ and $\nu\in\P\big(B_\Y(r)\big)$ that are the marginals of conditional plans $\pi|_A / \pi(A)$, where $\pi$ is optimal and $A\subset B_\X(r)\times B_\Y(r)$ has non-zero $\pi$-mass.
  
  \item \label{it:super-fast}
    \Cref{thm:unbounded-rates} does not allow for convergence rates $n^{-\alpha}$ with $\alpha > 1/2$ even if super-fast convergence was possible in Assumption~$\BC(\kappa, \alpha)$ (see equation \eqref{eq:cts_connected_bounds} for more context on super-fast convergence of empirical optimal transport).
    This is due to the fact that error terms inherent to our decomposition argument converge with $n^{-1/2}$ at best.

  \item \label{it:marginal-costs}
    There is a balancing issue that determines which marginal cost functions $\cX$ and $\cY$ are appropriate for application in \Cref{thm:unbounded-rates}.
    As a general rule: the larger the marginal costs are, the easier Assumption $\BC(\kappa, \alpha)$ is to satisfy, while the moment conditions become more restrictive.
  
  \item \label{it:independence}
    The independence assumption between the samples $(X_i)_{i=1}^n$ and $(Y_j)_{j=1}^m$ is crucial for the pursued proof strategy.
    Results in compact settings, in contrast, typically allow for arbitrary inter-sample dependencies.
    We consider this constraint in \Cref{thm:unbounded-rates} to be a likely artifact of our approach. In fact, an alternative proof strategy that we explore in \Cref{app:dual_decomposition}, but which is less general and requires Lipschitz cost functions, is free of this requirement.
    
  \item \label{it:uniform}
    The constant in equation \eqref{eq:unbounded-rates} can be expressed in terms of $\kappa$, $\epsilon$, $s$, $\alpha$, as well as the moments $\mint\mu \cX^{s+\epsilon}$ and $\mint\nu \cY^{s+\epsilon}$ (see \Cref{subsec:composition}).
    This means that inequality \eqref{eq:unbounded-rates} actually holds \emph{uniformly} over all measures whose $(s+\epsilon)$-th moments are bounded by a common constant.
    
\end{enumerate}

Since convergence rates in bounded settings are readily available in the literature, \Cref{thm:unbounded-rates} can easily be applied to derive novel results for a variety of problems.
For example, under costs $c(x, y) = \|x - y\|^p$ based on the Euclidean norm in $\RR^d$ for $p \ge 2$ and $d > 4$, we establish (\Cref{cor:wasserstein-rates})
\begin{subequations}%
\begin{equation*}
  \Exp\,\big|\T_c(\hat\mu_n, \hat\nu_n) - \T_c(\mu, \nu)\big|
  \lesssim
  n^{-\frar{2}{d}}
\end{equation*}
whenever $\mu\in\P(\RR^d)$ and $\nu\in\P(\RR^d)$ possess finite $\big(\frac{dp}{d-2}+\epsilon\big)$-th Euclidean moments.
This insight is novel in the sense that comparable results rely on much stronger concentration assumptions \parencite{manole2021sharp} or consider $\mu = \nu$ \parencite{fournier2015rate}.
If $0 < p < 2$ and $d > 2p$, on the other hand, the same corollary shows
\begin{equation*}
  \Exp\,\big|\T_c(\hat\mu_n, \hat\nu_n) - \T_c(\mu, \nu)\big|
  \lesssim
  n^{-\frar{p}{d}}
\end{equation*}
\label{eq:euclidean-rates}%
\end{subequations}%
under the assumption of finite $\big(\frac{dp}{d-p}+\epsilon\big)$-th moments, where similar comments apply.

The remainder of the article is structured as follows. In \Cref{sec:related}, we provide a thorough review of related literature.
In \Cref{sec:examples}, we identify a number of settings suited for application of \Cref{thm:unbounded-rates}.
Besides the Euclidean example mentioned above (\Cref{cor:wasserstein-rates}), we provide similar extensions to unbounded settings for more general spaces (\Cref{cor:wassersteinPolish}) and cost functions (Corollaries \ref{cor:lipschitz} and \ref{cor:smooth}).
In \Cref{sec:lower}, we discuss lower bounds and conclude that \Cref{thm:unbounded-rates} cannot generally be improved by much.
In particular, the required moment conditions are sharp up to the arbitrarily small term $\epsilon > 0$. \Cref{sec:proof} contains the proof of \Cref{thm:unbounded-rates} and provides justifications for the claims \ref{it:log} to \ref{it:uniform} formulated above. Finally, in \Cref{sec:discussion}, we discuss potential venues for future work, focussing on refinements of one-sample results as well as two-sample results under dependent samples. 

\section{Related work}
\label{sec:related}

\begin{table}[t]
  \centering
  \begin{tabular}{c@{\qquad}cccccccc}
    \toprule
    Contents $\backslash$ Article &  [BG14] & [FG15] & [WB19] & [MNW24] & [HSM24] &  This work\\[0.25em]
    \toprule
    {\small Sharp rates for $\mu = \nu$} &  \checkmark & \checkmark  & \checkmark &&& \checkmark \\
    {\small Sharp rates for $\mu \neq \nu$} & &  & & \checkmark & \checkmark& \checkmark\\
  \midrule
    {\small Wasserstein $p$-costs} & \checkmark & \checkmark & \checkmark & \checkmark & \checkmark& \checkmark\\
    {\small More general costs} & & & & \checkmark& \checkmark & \checkmark\\
  \midrule
    {\small Euclidean space} & \checkmark& \checkmark& \checkmark& \checkmark& \checkmark& \checkmark\\
    {\small More general spaces} & \checkmark && \checkmark&& \checkmark& \checkmark \\
  \midrule
    {\small Unbounded costs} & \checkmark& \checkmark&& \checkmark&& \checkmark \\
  \midrule
    {\small\begin{tabular}{@{}c@{}} Rates in terms of  \\  intrinsic dimension \end{tabular}}  & \checkmark & & \checkmark & & \checkmark & \checkmark \\
    \bottomrule
  \end{tabular}

  \caption{Overview of contents and limitations of selected contributions related to our work.
  \label{tab:Summary}}
\end{table}

The convergence of the empirical optimal transport cost to its population value has been the subject of extensive and long-standing research.
Below, we provide an overview of relevant contributions.
Given the vast volume of literature in this area, we mainly focus on work that establishes convergence rates comparable to the ones that we derive.
\Cref{tab:Summary} summarizes some recent contributions along with a comparison to our own work.

The most extensive line of related research has been dedicated to the question of how well an empirical measure $\hat \mu_n$ approximates the underlying population measure $\mu$ with respect to the \emph{Wasserstein distance}. The $p$-Wasserstein distance between probability measures $\mu, \nu\in \PC(\XC)$ on a Polish metric space $(\XC, d)$ for $p > 0$ is defined via $T_{d^p}(\mu, \nu)^{1/(p\vee 1)}$. It establishes a metric on the set of probability measures with finite $p$-th moments \parencite[Chapter 6]{villani2008optimal}, which accounts for its particular significance in geometric, computational, and statistical contexts.
\textcite{dudley1969speed} was the first to make a foundational contribution to the analysis of the $1$-Wasserstein distance $T_{d}(\hat \mu_n, \mu)$ on a general Polish metric space.
His result was extended to arbitrary $p\geq 1$ by \textcite{boissard2014mean} for Polish metric spaces with finite Minkowski dimension under high-order moment assumptions, and refined by \textcite{weed2019sharp} for compact metric spaces with a measure-dependent notion of dimension (the upper Wasserstein dimension). 
A comprehensive analysis for the Wasserstein distance on Euclidean spaces $\RR^d$ with Euclidean metric $\norm{x-y}$ was accomplished by \textcite{fournier2015rate}, who built  on coupling techniques by \textcite{dereich2013constructive}.
For $\mu \in \PC(\RR^d)$ with finite moment of order $q > p >0$, such that $q \neq 2p$ if $d\leq 2p$ and $q \neq dp/(d-p)$ if $d > 2p$, they prove for all $n\in \NN$ that\footnote{Note that Theorem 1 of \textcite{fournier2015rate} erroneously demands $q \neq d/(d-p)$ under $d > 2p$, but the assertion is actually shown for $q\neq dp/(d-p)$.}
\begin{equation}\label{eq:fournierBound}
  \EE\big[T_{\norm{\cdot}^p}(\hat \mu_n,\mu)\big] \lesssim \varphi_{p,d}(n)  + n^{\frab{q-p}{q}} \quad \text{ with } \quad  \varphi_{p,d}(n) \coloneqq \begin{cases}
  n^{-\frar{1}{2}} &  \text{ if } d < 2p,\\
  n^{-\frar{1}{2}}\log(n+1) &   \text{ if } d = 2p,\\
  n^{-\frar{p}{d}} &  \text{ if } d > 2p.
  \end{cases}
\end{equation}
The hidden constant only depends on $p,d,q$, and the $q$-th moment of $\mu$.
This upper bound can be shown to be nearly sharp in $d$ and $q$ for multiple examples (see \Cref{sec:lower} for the regime $d \neq 2p$).
It is even minimax rate optimal among all possible measure estimators based on i.i.d.\ random variables \parencite{singh2018minimax}. 
For $2 = d = 2p$ and $\mu = \Unif[0, 1]^2$, the celebrated Ajtai-Koml\'os-Tusn\'ady-matching theorem \parencite{ajtai1984optimal} states that $T_{\norm{\cdot}}(\hat \mu_n, \mu)\asymp n^{\frar{1}{2}}\log(n+1)^{\frar{1}{2}}$ with high probability (see also \cite{bobkov2021simple, divol2021short}).
This suggests that \eqref{eq:fournierBound} is not sharp for $d = 2p$ up to $\log(n+1)^{\frar{1}{2}}$, but it remains an open question whether the log-factor in \eqref{eq:fournierBound} can be improved in the general setting.
Moreover, refinements in terms of $q$ to treat the remaining cases of \eqref{eq:fournierBound} were derived by \textcite{dedecker2019behavior} and \textcite{lei2020convergence} and lead to additional logarithmic terms in the upper bound.
Notably, \textcite{lei2020convergence} also shows that the constants do not depend on the ambient dimension and derives convergence rates for measures on Banach spaces.
An explicit analysis of non-asymptotic constants for Euclidean settings was carried out by \textcite{kloeckner2020empirical, fournier2022convergence}. Moreover, parallel to our work, \textcite{larsson2023concentration} derived upper bounds as in \eqref{eq:fournierBound} for a class of homogenous Euclidean costs $c(x,y) = h(\norm{x-y})$ with $h(t) \asymp t^p$ for $t>0$ near zero and possibly super-polynomial growth for $t\rightarrow \infty$ under suitable integrability assumptions on the underlying measures. 

Improvements to \eqref{eq:fournierBound} can be made for $d\leq 2p$ with $p\geq 1$ if $\mu$ is absolutely continuous with connected support.
For instance, for $\mu = \Unif[0,1]^d$ it is known by \textcite{bobkov2019one} if $d= 1$, \textcite{ajtai1984optimal, talagrand1994matching} if $d = 2$, and \textcite{talagrand1994transportation} if $d \geq 3$, that for all $n\in \NN$,
\begin{equation}\label{eq:cts_connected_bounds}
	\EE\big[T_{\norm{\cdot}^p}(\hat \mu_n, \mu)\big] \asymp 
	\begin{cases}
 n^{-\frar{p}{2}}\log(n+1)^{\frar{p}{2}} & \text{ if } d = 2,\\
 n^{-\frar{p}{(d\vee 2)}} & \text{ if } d \neq 2. 	
 \end{cases}
\end{equation}
Remarkably, this result implies that faster convergence than the parametric $n^{-1/2}$ rate is possible on connected domains.
The asymptotic constant of \eqref{eq:cts_connected_bounds} as $n\rightarrow\infty$ was derived by \textcite{ambrosio2022quadratic} for $p = d = 2$ and by \textcite{goldman2021convergence} for $p\geq 1, d\geq 3$.
Similar assertions, up to additional logarithmic terms, remain valid if $\mu$ is replaced by a (multivariate) standard Gaussian distribution \parencite{ledoux2018optimal, ledoux2019optimal, ledoux2021optimal} or for $p = 2$ if $\mu$ admits a positively lower and upper bounded density on $[0,1]^d$ \parencite{manole2021plugin}. 
For the regime $d> 2p\geq 2$, no improvements to \eqref{eq:fournierBound} are possible if $\mu$ admits a density and has sufficiently many finite moments.
Indeed, in this setting, it was shown by \textcite{dobric1995asymptotics} for $p = 1$ and later extended by \textcite{barthe2013combinatorial} for $p \geq 1$ that almost surely
\begin{equation*}
  0
  <
  \liminf_{n\rightarrow \infty}n^{\frar{p}{d}}\,\T_{\norm{\cdot}^p}(\hat \mu_n, \mu)
  \leq
  \limsup_{n\rightarrow \infty}n^{\frar{p}{d}}\,\T_{\norm{\cdot}^p}(\hat \mu_n, \mu)
  <
  \infty. 
\end{equation*}

The literature cited so far had its focus on the convergence rate of the $p$-Wasserstein distance between an empirical measure and its population counterpart.
From a statistical perspective, however, convergence rates with respect to different population measures are of importance.
Invoking the triangle inequality for $\smash{\T_{d^p}^\frar{1}{p}}$ with $p \geq 1$, it follows from Jensen's inequality that
\begin{equation}\label{eq:naive_triangle_bound}
  \EE\,\big|\T_{d^p}(\hat \mu_n, \nu)^\frar{1}{p} - \T_{d^p}(\mu, \nu)^\frar{1}{p}\big|
  \leq
  \EE\big[\T_{d^p}(\hat \mu_n, \mu)^\frar{1}{p}\big]
  \leq
  \EE\big[\T_{d^p}(\hat \mu_n, \mu)\big]^\frar{1}{p}, 
\end{equation}
which asserts convergence rates at least as fast as if $\mu = \nu$.
In the Euclidean setting with sufficiently many moments, i.e., when the term $\varphi_{p,d}(n)$ dominates in equation \eqref{eq:fournierBound}, the resulting upper bound for equation \eqref{eq:naive_triangle_bound} is even minimax rate optimal up to logarithmic terms among all possible estimators for $\T_{\norm{\cdot}^p}(\mu, \nu)^\frar{1}{p}$ \parencite{niles2022estimation}. 
However, if a separability constraint of the form $\T_{\norm{\cdot}^p}(\mu, \nu) > \delta > 0$ is imposed on $\mu$ and $\nu$, the upper bound in \eqref{eq:naive_triangle_bound} can be improved.
This behavior was first observed by \textcite{sommerfeld2018inference} in the context of distributional limits for discrete population measures and later formalized for compactly supported probability measures on $\RR^d$ and $p = 2$ by \textcite{chizat2020faster}.
They established
\begin{equation}\label{eq:chizatbound}
  \EE\,\big|\T_{\norm{\cdot}^2}(\hat \mu_n, \nu)^\frar{1}{2} - \T_{\norm{\cdot}^2}(\mu, \nu)^\frar{1}{2}\big|
  \lesssim
  \varphi_{2,d}(n)
\end{equation}
where the implicit constant depends on the diameter of the support of $\mu$ and $\nu$, $\delta > 0$, and $d\in\NN$.
Compared to \eqref{eq:naive_triangle_bound}, this yields squared rates.
Bound \eqref{eq:chizatbound} was extended by \textcite{Deb2021} to the setting where one measure is sub-Weibull at the expense of additional logarithmic terms.

For more general ground cost functions, that are not necessarily based on a metric, the literature on convergence rates is far less copious.
For a class of costs of the form $c(x,y) = h(x-y)$ on Euclidean spaces $\RR^d$ with $d \geq 5$, \textcite{manole2021sharp} derived the upper bounds
\begin{equation}\label{eq:manole}
  \EE\,\big|\T_c(\hat\mu_n, \hat\nu_n) - \T_c(\mu, \nu)\big|
  \lesssim
  n^{-\alpha/d},
\end{equation}
where $0 < \alpha \le 2$ denotes the Hölder-smoothness of the cost function.
The measures $\mu$ and $\nu$ are either required to be compactly supported or absolutely continuous with strong concentration (sub-Weibull) as well as certain anti-concentration properties.
The upper bound \eqref{eq:manole} is sharp if $\mu$ and $\nu$ are absolutely continuous.
For compactly supported measures, later work by \textcite{hundrieser2022empirical}  showed that the right-hand side of \eqref{eq:manole} generalizes to $\varphi_{\alpha, d}(n)$ for $d \ge 1$ and a more generic class of Hölder-smooth costs.
In particular, their results also cover some non-Euclidean settings, like metric spaces with finite Minkowski-Bouligand dimension or compact smooth manifolds.
Moreover, \textcite{hundrieser2022empirical} documented the general phenomenon of \emph{lower complexity adaptation} of empirical optimal transport, i.e., that convergence rates of $\T_c(\hat\mu_n, \hat\nu_n)$  towards the population quantity $\T_c(\mu, \nu)$ are determined by the less complex measure out of $\mu$ and $\nu$.
Instances of this phenomenon have previously been reported in semi-discrete settings. For example,
\textcite{Forrow19} showed that the convergence rates in the Euclidean $p$-Wasserstein setting \eqref{eq:chizatbound} are of parametric order $n^{-1/2}$ if $\nu$ is discrete and $\mu$ is compactly supported.
Similar observations were made by \textcite{del2022central} for the 1-Wasserstein distance on more general metric spaces.

\section{Applications of the main theorem}
\label{sec:examples}

This section elaborates on a number of examples that illustrate how \Cref{thm:unbounded-rates} can be applied in prevalent settings. 
We first concentrate on the $p$-Wasserstein distances on Euclidean and Polish spaces.
Afterwards, we treat more general costs that are either Lipschitz-continuous or smooth in one component. Here, we confine ourselves to (partially) Euclidean settings $\X\subset\RR^d$ for some $d\in\NN$, even though more general results, e.g., in geodesic spaces, can likely be derived as well.
An important consequence of our considerations is that lower complexity adaptation (see \cite{hundrieser2022empirical}) also emerges in unbounded settings.
For the sake of simplicity, all statements in this section are formulated for equal sample sizes $m = n$. The generalization to $m \neq n$ by replacing $n$ in the upper bounds with $n \wedge m$ is straightforward.

\subsection{Wasserstein distances on Euclidean spaces}
\label{subsec:WassersteinEucl}

Let $\X = \Y = \RR^d$ for $d\in\NN$ and let $\B = \{x\in\RR^d\,|\,\|x\| \le 1\}$
be the Euclidean unit ball. Consider the costs $c(x, y) = \|x - y\|^p$ for $p > 0$.
Setting $c_\X = c_\Y = 2^{p}\|\cdot\|^p$, we note $c \le \cX \oplus\cY$ and
\begin{equation*}
  B_\X(r)
  =
  B_\Y(r)
  =
  \big\{x\in\RR^d\,\big|\, \|x\| \leq 2^{-1}r^{1/p}\big\}
  =
  2^{-1}r^{1/p}\,\B
\end{equation*}
for any $r\geq 1$.
From the literature on empirical optimal transport on compacta \parencite{fournier2015rate, manole2021sharp, hundrieser2022empirical}, we know that
\begin{equation}\label{eq:upperBoundEmpiricalOT}
    \Exp\,\big| \T_{\norm{\cdot}^p}(\hat\mu_n, \hat\nu_n) - \T_{\norm{\cdot}^p}(\mu, \nu) \big|
  \lesssim
  \begin{cases}
    \varphi_{p,d}(n)  & \text{ if } \mu = \nu,\\
    \varphi_{p \wedge 2,d}(n)& \text{ if } \mu \neq \nu,
  \end{cases}
\end{equation}
for any $\mu, \nu \in \PC(\B)$ with the rate function $\varphi_{p,d}$ defined in \eqref{eq:fournierBound} and a constant depending on $p$ and $d$ only.
We next observe that any measure with support in $B_\X(r)$ can be rescaled via a pushforward along the map $u(x) = 2x/r^{1/p}$ such that it is supported in $\B$.
This affects the cost function by a scaling only.
Therefore, we can apply \eqref{eq:upperBoundEmpiricalOT} to any $\mu\in \P\big(B_\X(r)\big)$ and $\nu\in \P\big(B_\Y(r)\big)$ for $r\geq 1$ via
\begin{align}
  \Exp\,\big| \T_{\norm{\cdot}^p}(\hat\mu_n, \hat\nu_n) - \T_{\norm{\cdot}^p}(\mu, \nu) \big|
  &=
  2^{-p} r\,\Exp\,\big| \T_{\norm{\cdot}^p}(u_\#\hat\mu_n, u_\#\hat\nu_n) - \T_{\norm{\cdot}^p}(u_\#\mu, u_\#\nu) \big| \nonumber \\
  &\le
  \kappa\,r\, \begin{cases}
    \varphi_{p,d}(n) & \text{ if } \mu = \nu,\\
    \varphi_{p \wedge 2,d}(n)  & \text{ if } \mu \neq \nu,
  \end{cases} \label{eq:wasserstein-p-bc}
\end{align}
where $\kappa$ depends on $p$ and $d$ alone.
This establishes the bounded convergence assumption that is necessary for the application of \Cref{thm:unbounded-rates}.
\begin{corollary}{Euclidean Wasserstein rates}{wasserstein-rates}
  Let $\X = \Y = \RR^d$ and $c(x, y) = \|x - y\|^p$ for $p > 0$ and $d\in\NN$.
  Assume $\mu, \nu\in\P(\RR^d)$ possess finite $(q+\epsilon)$-th Euclidean moments for $p < q \le 2p$ and $\epsilon > 0$.
  Then, for any $n\in\NN$, it holds that
  \begin{equation*}
    \Exp\,\big| \T_{\norm{\cdot}^p}(\hat\mu_n, \hat\nu_n) - \T_{\norm{\cdot}^p}(\mu, \nu) \big|
    \lesssim
    \begin{cases}
      \varphi_{p,d}(n) \;\,\,\,+ n^{-\frab{q-p}{q}} & \text{ if } \mu = \nu,\\
      \varphi_{p\wedge 2,d}(n) + n^{-\frab{q-p}{q}} & \text{ if } \mu \neq \nu,
    \end{cases}
  \end{equation*}
  where the constant depends on $\epsilon$, $p$, $q$, $d$, and the $(q+\epsilon)$-th Euclidean moments of $\mu$, $\nu$. 
\end{corollary}
\begin{proof}
  The existence of $(q+\epsilon)$-th Euclidean moments is equivalent to $\mint\mu \cX^{s+\epsilon'} < \infty$ and $\mint\nu \cY^{s+\epsilon'} < \infty$ for $1 < s = q/p \le 2$ and $\epsilon' = \epsilon/p > 0$.
  Since Assumption $\BC(\kappa, \alpha)$ is satisfied for suitable $\kappa$ and $\alpha$ via \eqref{eq:wasserstein-p-bc}, the result follows due to \Cref{thm:unbounded-rates}.
  The additional $\log(n+1)$ term in $\varphi_{p,d}$ for $p = 2d$ is handled via comment \ref{it:log} below the theorem.
  Similarly, the dependency of the implicit constant follows from comment \ref{it:uniform}. 
\end{proof}

If sufficiently many moments exist, for example $q \ge \frac{dp}{d-p}$ for $d \geq 2p$ if $\mu = \nu$ or $q \ge \frac{dp}{d-(p \wedge 2)}$ for $d \ge 2(p \wedge 2)$ if $\mu \neq  \nu$, we conclude that the rates from \eqref{eq:upperBoundEmpiricalOT} generalize to the unbounded setting.
Compared to the upper bounds by \textcite{fournier2015rate}, which only apply for $\mu = \nu$, the moment conditions in \Cref{cor:wasserstein-rates} are loose by the term $\epsilon > 0$. Compared to the bounds by \textcite{manole2021sharp}, our result does not require sub-exponential tails of the underlying measures to ensure the rates known from compact settings, as it does not rely on quantitative bounds for the (empirical) dual solutions.
The lower bounds presented in \Cref{sec:lower} assert that the rates derived in \Cref{cor:wasserstein-rates} are (almost) sharp over the collection of probability measures with finite $(q+\epsilon)$-th Euclidean moment on $\RR^d$.
\begin{remark}{}{}
  There are instances where the upper bounds in \Cref{cor:wasserstein-rates} are not sharp.
  If $\mu = \nu$, for example, equation \eqref{eq:cts_connected_bounds} suggests faster-than-parametric rates on connected domains if $p > d/2$. 
  As mentioned in comment \ref{it:super-fast} below \Cref{thm:unbounded-rates}, this cannot be captured by our decomposition strategy. 
  Another well-known special case concerns $p=1$ and $d=2$, for which the sharp rate $(n\log(n+1))^{1/2}$ was established by \cite{talagrand1992ajtai} in compact settings and extended by \cite{yukich1992generalizations} to unbounded ones under suitable concentration assumptions.
  This improves \Cref{cor:wasserstein-rates} for $q=2$ by a factor of $\log(n+1)^{1/2}$.
  Our decomposition strategy is applicable in this case, however, and we can recover the rates by \cite{yukich1992generalizations} on the basis of \cite{talagrand1992ajtai}.
\end{remark}

\subsection{Wasserstein distances on Polish spaces}
\label{subsec:WassersteinPol}

The results of the previous section can be extended to Polish spaces that are finite dimensional in a suitable sense.
To this end, let $(\YC, d)$ be a Polish metric space and let $\XC$ be a Polish subset of $\YC$.
We pick this arrangement of $\XC$ and $\YC$ in order to account for lower complexity adaptation, which manifests itself if the support of $\mu$ is \enquote{less complex} than the one of $\nu$.
In particular, we only have to be able to control the complexity of $\XC$, not the one of $\YC$.

We consider the cost function $c(x,y) = d^p(x,y)$.
It suffices to consider $p\geq 1$, since one may replace the underlying ground metric by $d' \coloneqq d^p$ if $0 < p < 1$.
For a fixed element $x_0\in \XC$, let $\cX= 2^{p} d^p(\cdot, x_0)$ and $\cY = 2^{p} d^p(\cdot, x_0)$. By the triangle inequality, we observe $c\leq \cX\oplus \cY$. The corresponding $\cX$- and $\cY$-balls are given by
\begin{equation*}
  B_\X(r)
  =
  \big\{x\in\XC\,\big|\, d(x,x_0) \leq  2^{-1}r^{1/p}\big\}
  \qquad\text{and}\qquad
  B_\Y(r)
  =
  \big\{y\in\YC\,\big|\, d(y,x_0) \leq  2^{-1}r^{1/p}\big\}.
\end{equation*}
To analyze the convergence of $\smash{T_c(\hat \mu_n, \hat \nu_n)}$ to $T_c(\mu, \nu)$ in the $p$-Wasserstein distance, we impose a condition of dimensionality on the space $(\XC, d)$.
This requires the notion of \emph{covering numbers}.
For a set $U\subseteq \XC$, the $\delta$-covering number $\covering(\delta, U, d)$ of $U$ for $\delta>0$ is defined as the smallest integer $N\in \NN$ such that there are $x_1, \dots, x_n\in \XC$ with $U \subseteq \bigcup_{i = 1}^N B(x_i, \delta)$, where $B(x_i, \delta)\coloneqq \{x\in \XC \,|\, d(x,x_i) \leq \delta\}$ denotes the metric ball with radius $\delta$ around $x_i$. 

\begin{assumption*}{$\DIM(t)$}{}
  Let $t > 0$.
  Assume there exists a constant $k_\XC\geq 1$ such that any $U\subseteq \XC$ with diameter $\diam(U) \coloneqq \sup_{x,y\in U} d(x,y)$ satisfies, for any $\delta > 0$,
  \begin{equation*}
    \covering(\delta, U, d)
    \leq
    1+k_\XC \left(\frac{\diam(U)}{\delta}\right)^t.
  \end{equation*}
\end{assumption*}
This condition was introduced by \textcite{boissard2014mean} and places an upper bound on the \emph{Minkowski-dimension} of the metric space $(\XC, d)$.
For instance, $\DIM(t)$ is fulfilled if $(\XC, d) = (\RR^t, \|\cdot\|)$ with $t\in \NN$.
Assumption $\DIM(t)$ is also fulfilled if the space $\XC$ is a sufficiently regular $t$-dimensional sub-manifold of a Euclidean space equipped with the geodesic distance. 
Notably, the quantity $t > 0$ does not need to be integer-valued. 

If Assumption~$\DIM(t)$ holds for $\mu = \nu$ and $p \ge 1$, it follows from the proof of Proposition~5 in \textcite{weed2019sharp} for any $r\geq 1$ and $\mu\in\P\big(B_\XC(r)\big)$ that
\begin{subequations}
\label{eq:ConvergenceRateWassersteinPolish}
\begin{equation}\label{eq:ConvergenceRateWassersteinNull}
  \Exp\,\big| \T_{d^p}(\hat\mu_n, \hat\nu_n) \big|
  \lesssim
  r \varphi_{p,t}(n) 
\end{equation}
with a constant depending on $p$, $t$, and $k_\XC$ only.
For $\mu \neq \nu$ and $p = 1$, the dual formulation of optimal transport in conjunction with metric entropy bounds by \textcite{Kolmogorov1961} and methods laid out by \textcite{hundrieser2022empirical} can be used to show
\begin{equation}\label{eq:ConvergenceRateWassersteinAlternative}
  \Exp\,\big| \T_{d}(\hat\mu_n, \hat\nu_n) - T_{d}(\mu, \nu) \big| 
  \lesssim
  r	\varphi_{1,t}(n)
\end{equation}
\end{subequations}
for $r\geq 1$, $\mu\in \PC\big(B_\XC(r)\big)$, and $\nu \in \PC\big(B_\YC(r)\big)$ if $B_\XC(r)$ is connected (see \Cref{rem:connected} below).
The constant only depends on $t$ and $k_\XC$.
It is open whether a similar bound remains valid for the remaining regime $\mu \neq \nu$ and $p > 1$.
We formally prove inequalities \eqref{eq:ConvergenceRateWassersteinPolish} in \Cref{app:auxiliaryProofs}.

\begin{corollary}{Polish Wasserstein rates}{wassersteinPolish}
  Let $\X$ be a Polish subset of a Polish metric space $(\Y, d)$ such that $\X$ satisfies $\DIM(t)$ and let $c(x, y) = d^p(x,y)$ for $p \geq 1$.
  Assume $\mu\in\P(\XC)$ and $\nu \in \PC(\YC)$ possess finite $(q+\epsilon)$-th moments $\mint{(\mu + \nu)}d^{(q+\epsilon)}(\cdot,x_0)<\infty$ for $p < q \le 2p$ and $\epsilon > 0$.
  If either $\mu = \nu$ and $p \geq 1$, or $\mu \neq \nu$ and $p=1$ with $B_\XC(r)$ connected for any $r\geq 1$, then, for all $n\in\NN$,
  \begin{equation}\label{eq:wassersteinPolish}
    \Exp\,\big| \T_{d^p}(\hat\mu_n, \hat \nu_n) - T_{d^p}(\mu, \nu)\big| 
    \lesssim
    \varphi_{p,t}(n) + n^{-\frab{q-p}{q}},
  \end{equation}
where the constant only depends on $t$, $p$, $\epsilon$, $k_\XC$, and the $(q+\epsilon)$-th moments of $\mu$ and $\nu$.
\end{corollary}

\begin{proof}
Similarly to \Cref{cor:wasserstein-rates}, the result follows from \Cref{thm:unbounded-rates} together with the convergence rates in \eqref{eq:ConvergenceRateWassersteinPolish} and the imposed moment conditions.
\end{proof}

\begin{remark}{}{connected}
  It is sufficient in \eqref{eq:ConvergenceRateWassersteinAlternative}, and thus \eqref{eq:wassersteinPolish}, that $B_\XC(r)$ has a finite number of connected components uniformly over $r$.
Moreover, without imposing any conditions on the topology of $B_\XC(r)$, the upper bound \eqref{eq:wassersteinPolish} remains valid if $\varphi_{1,t}$ is replaced by
  \begin{equation*}
    \tilde \varphi_{1,t}(n)
    =
    \begin{cases}
      \varphi_{1, t}(n)                     & \text{if}~t < 2,     \\
    \varphi_{1,t}(n)\log(n+1) & \text{if}~t \ge 2.
    \end{cases}
  \end{equation*}
  This generalization is also detailed in \Cref{app:auxiliaryProofs}.
\end{remark}

In case of identical measures $\mu = \nu$, \Cref{cor:wassersteinPolish} captures the full regime $t > 0$ and $q > p$ (where one cannot improve over $q = 2p$, which already guarantees the same rates as in the compact setting).
In comparison, \textcite{boissard2014mean} require $t > p$ and $q > 2p \vee (tp / (t-p))$ for analog convergence rates, and the results by \textcite{weed2019sharp}, which work with a different notion of dimensionality, are only derived for compact spaces.
For different measures $\mu \neq \nu$, \Cref{cor:wassersteinPolish} reveals lower complexity adaptation: the bound in \eqref{eq:wassersteinPolish} is independent of the Polish space $\YC$.
In the semi-discrete setting, i.e., if $\mu$ is supported on a finite set $\XC\subseteq \YC$ while $\nu$ can be a general probability measure on $\YC$, it follows that $\DIM(t)$ is met for any $t > 0$. Hence, if the $(q+\epsilon)$-th moment of $\nu$ for $q>1$ and $\epsilon>0$ is finite, \Cref{cor:wassersteinPolish} asserts
\begin{align*}
  \Exp\,\big| \T_{d}(\hat\mu_n, \hat \nu_m) - T_{d}(\mu, \nu)\big| 
\lesssim
n ^{-\frar{1}{2}} + n^{-\frab{q-1}{q}}.
\end{align*}
This is in line with recent findings by \textcite{del2022central} for $p = 1$, who prove the parametric $n^{-1/2}$-rate if $\nu$ possesses a finite second moment. 
Section \ref{sec:lower} illustrates that the convergence rates achieved by \Cref{cor:wassersteinPolish} are nearly optimal for appropriate metric spaces $\XC\subseteq \YC$ when considering probability measures with a finite moment of order $q+\epsilon$. 
Again, the moment conditions are only loose by $\epsilon > 0$.

\subsection{Locally Lipschitz costs}
\label{subsec:LocallyLipschitz}

We now depart from the Wasserstein setting and look at a class of generic locally Lipschitz cost functions. The space $\X = \RR^d$ for $d\in\NN$ is assumed to be Euclidean while $\Y$ can be general Polish.
Like before, $\B \subset \RR^d$ denotes the Euclidean unit ball.
A suitable bounded convergence result for the application of \Cref{thm:unbounded-rates} is provided by \textcite{hundrieser2022empirical}.
For a Borel set $U\subseteq \YC$ and a continuous cost function $c$, for which the partially evaluated costs $c(\cdot, y)$ are bounded by $1$ and $1$-Lipschitz continuous on $\B$ for all $y\in U$, they showed
\begin{equation}\label{eq:lipschitz-bc}
    \Exp\,\big| \T_c(\hat\mu_n, \hat\nu_n) - \T_c(\mu, \nu) \big|
  \lesssim
  \varphi_{1,d}(n)
\end{equation}
for any measures $\mu\in \PC(\B)$ and $\nu \in \PC(U)$.
The constant in this inequality only depends on the dimension $d$.
We introduce the following requirement to characterize an appropriate class of locally Lipschitz costs.
\begin{assumption*}{$\LIP(p)$}{}
  Let $c$ be of form \eqref{eq:marginal-cost-bound} with $\cX(x) = \|x\|^p$ for $p > 0$ and $\cY\colon\Y\to\RRplus$ lower semicontinuous.
  Assume that $c(\cdot, y)$ is locally Lipschitz with first partial derivatives at $x$, if they exist, bounded by $\|x\|^{p-1} + r^{1-1/p}$ for all $x\in B_\X(r)$ and $y\in B_\Y(r)$.
\end{assumption*}
This assumption demands that the local Lipschitz constant of $c$ in its first argument does not grow too quickly if $p \ge 1$, or that it decays sufficiently fast if $p < 1$.
In Euclidean settings with $\Y = \RR^d$, for example, $\LIP(p)$ can be satisfied\footnote{
  One needs to rescale the cost function $c$ suitably and can then, e.g., pick $c_\Y(y) = 1 + \|y\|^p$.
} for any $p > 0$ if $c$ is locally Lipschitz with $c(x, y) \lesssim \smash{\big(1 + \|x\| + \|y\|\big)^p}$ and $\|\nabla_x c(x, y)\| \lesssim \smash{\big(1 + \|x\| + \|y\|\big)^{p-1}}$ Lebesgue-almost everywhere.
In particular, this includes the Wasserstein costs $c(x, y) = \|x - y\|^p$ for $p \ge 1$.
Note that even oscillating costs can be treated within our theory, at the expense of stricter moment conditions (i.e., large $p$) if the derivative of $c$ in its first component dominates.

\begin{corollary}{Lipschitz rates}{lipschitz}
  Let $\X = \RR^d$, $\Y$ be Polish, and $c\colon\XY\to\RRplus$ continuous such that $\LIP(p)$ is satisfied.
  If $\mu\in\P(\RR^d)$ possesses finite $(q+\epsilon)$-th Euclidean moments and $\nu\in\P(\Y)$ fulfills $\mint\nu \smash{\cY^{q/p+\epsilon}} < \infty$ for $p < q \le 2p$ and $\epsilon > 0$, then, for any $n\in\NN$,
  \begin{equation}\label{eq:LipschitzRate}
    \Exp\,\big| \T_c(\hat\mu_n, \hat\nu_n) - \T_c(\mu, \nu) \big|
    \lesssim
    \varphi_{1,d}(n) + n^{-\frab{q-p}{q}},
  \end{equation}
  where the constant only depends on $d$, $q$, $p$, $\epsilon$, as well as the moments of $\mu$ and $\nu$.
\end{corollary}
\begin{proof}
  We have to establish the bounded convergence condition $\BC(\kappa, \alpha)$ in this setting.
  To this end, fix $r \ge 1$.
  We consider the (nonlinear) scaling function $u(x) = \|x\|^{p/\beta - 1}x / r^{1/\beta}$ for some $\beta > \max(1, p)$, which can be checked to map $B_\X(r)$ onto the Euclidean unit ball $\B$.
  According to part 1 and 2 of \Cref{lem:scaled-function}, the scaled cost function $c_r\colon\B\times B_{\Y}(r) \to \RRplus$ defined by $c_r(\cdot, y) = \smash{\frac{p}{2d\beta r}} \,c\big(u^{-1}(\cdot), y\big)$ is bounded by $1$ and is $1$-Lipschitz in its first component for all $y\in B_{\Y}(r)$ due to Assumption~$\LIP(p)$.
  Let $\mu\in\P\big(B_\X(r)\big)$ and $\nu\in\P\big(B_\Y(r)\big)$.
  We conduct a change of variables along $u$ and find
  \begin{equation*}
    \Exp\,\big|\T_c(\hat\mu_n, \hat\nu_n) - \T_c(\mu, \nu)\big|
    =
    \frac{2d\beta r}{p}\,\Exp\,\big|\T_{c_r}(u_\#\hat\mu_n, \hat\nu_n) - \T_{c_r}(u_\#\mu, \nu)\big|
    \le
    \kappa r\,\varphi_{1,d}(n),
  \end{equation*}
  where we have exploited that $u_\#\hat\mu_n$ is the empirical measure of $u_\#\mu \in \P(\B)$.
  The final inequality follows from \eqref{eq:lipschitz-bc} with a $\kappa > 0$ that only depends
  on $d$ and $p$.
  The claim of the corollary is now a direct consequence of \Cref{thm:unbounded-rates} for $s = q / p$, as well as comment \ref{it:uniform} below the theorem for the nature of the constants.
\end{proof}

\Cref{cor:lipschitz} can easily be used to derive convergence results in settings where one measure is concentrated on a $t$-dimensional surface within a $d$-dimensional ambient space for $t < d$.
In this case, the smaller dimension $t$ governs the convergence rate.

\begin{example}{Lipschitz surfaces}{lipschitz-surface}
  Let $\XC$ be a Polish subset of $\YC=\RR^d$ for $d\in \NN$ and let $c\colon \RR^d \times \RR^d \rightarrow \RRplus$ satisfy $\Lip(p)$ for $\smash{c_{\XC}} =  c_\YC = \norm{\cdot}^p$ with $p > 0$.
  Assume that $\XC$ lies within the image of a bi-Lipschitz function $f\colon \RR^t \to \RR^d$, where $t < d$.
  It is straightforward to see that $c_f(x,y) = c\big(f(x),y\big)$ fulfills $\Lip(p)$ up to a scaling that depends on $f$ and $p$ only.
  Moreover, for any $\mu \in \PC(\XC)$ and $\nu \in \PC(\YC)$,
  \begin{equation*}
    \T_{c}(\mu, \nu) = \T_{c_f}(f^{-1}_{\#}\mu, \nu).
  \end{equation*}
  Since $f^{-1}$ is Lipschitz, moments of $\mu$ are finite if and only if the same moments of $f^{-1}_{\#}\mu$ are finite.
  We can thus conclude that \Cref{cor:lipschitz} asserts the convergence rate \eqref{eq:LipschitzRate} with dimension $t$ instead of $d$ for $T_{c}(\hat \mu_n, \hat \nu_n)$ if both $\mu$ and $\nu$ admit finite Euclidean moments of order $(q+\epsilon)$ with $q>p$ and $\epsilon>0$. 
\end{example}

Only few results that are comparable to \Cref{cor:lipschitz}, in the sense that they cover convergence rates for a generic class of locally Lipschitz cost functions in unbounded settings, have been established so far.
Most notably, the work of \textcite{manole2021sharp} requires the specific Euclidean cost structure $c(x, y) = h(x - y)$ for a suitable function $h$, and strong (sub-exponential) concentration properties of the measures $\mu$ and $\nu$.

\subsection{Smooth costs}
\label{subsec:SmoothCosts}

A similar strategy as in the previous section can be employed to derive improved convergence rates for costs that are not only Lipschitz continuous but additionally semi-concave in one component.
For convenience, we work with cost functions that are twice differentiable in their first argument: bounded second derivatives imply semi-concavity.
We again let $\X = \RR^d$ while $\Y$ can be arbitrary Polish, and we write $\B$ to denote the Euclidean unit ball.
For a Borel set $U\subseteq \YC$ such that all partial derivatives of $c(\cdot, y)$ up to order $2$ are bounded by $1$ on $\B$ for all $y\in U$, it is known \parencite{hundrieser2022empirical}, for any $\mu\in\PC(\B)$ and $\nu\in \P(U)$, that
\begin{equation}\label{eq:smooth-bc}
    \Exp\,\big| \T_c(\hat\mu_n, \hat\nu_n) - \T_c(\mu, \nu) \big|
  \lesssim
  \varphi_{2,d}(n),
\end{equation}
where the implicit constant only depends on $d$.
Note that \eqref{eq:smooth-bc} is a strict improvement over bound \eqref{eq:lipschitz-bc} whenever $d > 1$.

\begin{assumption*}{$\SMOOTH(p)$}{}
  Let $c$ be of form \eqref{eq:marginal-cost-bound} with $\cX(x) = \|x\|^p$ for $p > 0$ and $\cY\colon\Y\to\RRplus$ lower semi-continuous.
  Assume that $c(\cdot, y)$ is twice differentiable with partial derivatives of order $k\in\{1,2\}$ at $x$ bounded by $\|x\|^{p-k} + r^{1-k/p}$ for all $r \ge 1$, $x\in B_\X(r)$, $y\in B_\Y(r)$.
\end{assumption*}

Similar to Assumption $\LIP(p)$ of the previous section, this condition states that the magnitude of the derivatives of $c$ can have an influence on the moment requirements when applying \Cref{thm:unbounded-rates}.

\begin{corollary}{smooth rates}{smooth}
  Let $\X = \RR^d$, $\Y$ be Polish, and $c\colon\XY\to\RRplus$ continuous such that
  $\SMOOTH(p)$ is satisfied. If $\mu\in\P(\RR^d)$ possesses
  finite $(q+\epsilon)$-th Euclidean moments and $\nu\in\P(\Y)$ fulfills
  $\mint\nu \smash{\cY^{q/p+\epsilon}} < \infty$ for $ p<q\leq 2p$ and $\epsilon > 0$, then, for any $n\in\NN$,
  \begin{equation}\label{eq:smooth-rate}
    \Exp\,\big| \T_c(\hat\mu_n, \hat\nu_n) - \T_c(\mu, \nu) \big|
  \lesssim
  \varphi_{2,d}(n) + n^{-\frab{q-p}{q}},
  \end{equation}
  where the constant only depends on $d$, $q$, $p$, $\epsilon$, as well as the moments of $\mu$ and $\nu$.
\end{corollary}
\begin{proof}
  The proof proceeds analogously to the one of \Cref{cor:lipschitz}.
  We only have to replace the scaled cost function by $c_r(\cdot, y) = \smash{\frac{p^2}{5d\beta^2r}}\,c\big(u^{-1}(\cdot), y\big)$ for $\beta > \max(2, 2p)$, for which \Cref{lem:scaled-function} together with Assumption $\SMOOTH(p)$ guarantees that $c_r(\cdot, y)$ and its derivatives up to order $2$ are bounded by one for any $y\in B_\YC(r)$. 
\end{proof}

\Cref{ex:lipschitz-surface} can be extended via \Cref{cor:smooth} to the setting where one probability measure is concentrated on a low-dimensional smooth surface.

\begin{example}{smooth surfaces}{smooth-surface}
  Consider the setting of \Cref{ex:lipschitz-surface} and additionally assume that $c$ fulfills $\SMOOTH(p)$.
  Further, assume that $f\colon \RR^t\rightarrow \RR^d$ has uniformly bounded second derivatives.
  Then the cost function $c_f$ also fulfills $\SMOOTH(p)$ up to scaling.
  Hence, by \Cref{cor:smooth}, upper bound \eqref{eq:smooth-rate} holds with dimension $t$ instead of $d$ if $\mu$ and $\nu$ both admit finite moments of order $(q+\epsilon)$ with $q>p$ and $\epsilon>0$. 
\end{example}

\section{Lower bounds}
\label{sec:lower}

We now summarize known results for lower bounds on the convergence of the empirical optimal transport cost, demonstrating that the rates obtained by \Cref{thm:unbounded-rates} are essentially sharp in various regimes.
We differentiate between parametric lower bounds induced by basic duality arguments, lower bounds induced by the domain dimensionality, and lower bounds induced by the tail behavior of the underlying measure. 

\subsection{Parametric lower bound via duality}
\label{subsec:lowerbound:duality}
For a lower semi-continuous cost function $c\colon \XC\times \YC \rightarrow \RRplus$ of the form \eqref{eq:marginal-cost-bound} with $\mint\mu c_\XC + \mint\nu c_\YC<\infty$, the optimal transport cost admits the dual formulation \parencite[Theorem 5.10]{villani2008optimal} 
\begin{align}\label{eq:dual}
  T_c(\mu, \nu)
  =
  \max_{\substack{(f,g) \in \Phi_c(\mu, \nu)}} \mint\mu f + \mint\nu g = \max_{\substack{(f,g) \in \Phi_c(\mu,\nu)}} \int f(x) \dif\mu(x) + \int g(y)  \dif\nu(y),
\end{align}
where $\Phi_c(\mu, \nu) = \big\{ (f,g) \in L^1(\mu) \times L^1(\nu) \,\big|\, f\oplus g \leq c\big\}$.
Dual solutions $(f,g)\in \Phi_c(\mu, \nu)$ for this optimization problem always exist.
Since $f$ and $g$ are $\mu$- and $\nu$-almost surely finite, respectively, it follows that $(f,g)\in \Phi_c(\hat \mu_n, \hat \nu_m)$ is satisfied almost surely, asserting that
\begin{align*}
  T_c(\hat \mu_n, \hat\nu_n) - T_c(\mu, \nu)
  \geq
  \mint{(\hat \mu_n - \mu)}f + \mint{(\hat \nu_n - \nu)}g
  \qquad\text{and}\qquad
  T_c(\hat \mu_n, \nu) - T_c(\mu, \nu)
  \geq
  \mint{(\hat \mu_n - \mu)}f.
\end{align*}
As long as $f$ and $g$ are not almost surely constant, this observation delivers lower bounds with parametric rates.
If $\mint{(\mu \otimes \nu)} c^2<\infty$, for example, \Cref{lem:lp-potentials} attests that $(f,g) \in L^2(\mu)\times L^2(\nu)$.
Then, the standard central limit theorem for independent empirical measures implies
\begin{subequations}
\label{eq:lowerbound_dual_combined}
\begin{align}
\label{eq:lowerbound_dual_one}
  \EE\,\big|T_c(\hat \mu_n, \hat\nu_n) - T_c(\mu, \nu)\big|&\gtrsim n^{-\frar{1}{2}}\sqrt{\Var_{X\sim \mu}[f(X)]+\Var_{Y\sim \nu}[g(Y)]},\\
  \label{eq:lowerbound_dual_two}
   \EE\,\big|T_c(\hat \mu_n, \nu) - T_c(\mu, \nu)\big|&\gtrsim n^{-\frar{1}{2}}\sqrt{\Var_{X\sim \mu}[f(X)]},
\end{align}
\end{subequations}
yielding a lower bound of parametric order $n^{-1/2}$.

Constancy of dual solutions has been thoroughly analyzed by \textcite[Section~4]{hundrieser2022unifying}.
If the measures $\mu$ and $\nu$ are different, constant dual solutions only arise in settings with a special symmetry, where optimal transport maps have a very simple structure and coincide with specific projections from the support of one measure onto the other.
Thus, non-constancy is the norm.
For identical measures $\mu = \nu$, constant dual solutions are more common.
However, under symmetric and positive definite cost functions, non-constant solutions always exists if the support of $\mu$ is split into two (or more) well-separated components \parencite[Lemma~11]{staudt2022uniqueness}.
For example, this holds if $\mu$ is supported on finitely many points.

\subsection{Lower bounds induced by dimensionality}
\label{subsec:lowerbound:dim}

In high dimensions, the bias of the empirical optimal transport cost typically converges at a speed strictly slower than the parametric $n^{-1/2}$ rate obtained in the previous section.
This dimension-dependent behavior is captured by various lower bounds that have been established in the literature.
Since
\begin{equation*}
  \EE\left[T_c(\hat \mu_n, \hat\nu_n) - T_c(\mu, \nu)\right]
  \geq
  \EE\left[T_c(\hat \mu_n, \nu) - T_c(\mu, \nu)\right],
\end{equation*}
which follows from the dual formulation \eqref{eq:dual} and is proven as point \ref{it:one-sample} at the end of \Cref{subsec:composition}, it suffices to consider the one-sample case. 

We first consider Wasserstein costs $c(x,y) = \norm{x-y}^p$ with $p\geq 1$ in Euclidean space $\RR^d$.
In case of identical uniform measures $\mu = \nu = \Unif[0,1]^d$ and $p \ge 1$, it is known by \textcite{dudley1969speed} for $d\neq 2$ and \textcite{ajtai1984optimal} for $d = 2$ that 
\begin{align*}
  \EE\big[T_{\norm{\cdot}^p}(\hat \mu_n, \mu)\big]
  \gtrsim 
  \begin{cases}
     n^{-\frar{p}{2}}\log(n+1)^{\frar{p}{2}} & \text{ if } d = 2,\\
     n^{-\frar{p}{d}} & \text{ if } d \geq 3,
  \end{cases}
\end{align*}
for all $n\in\NN$.
In contrast, if $\mu = \Unif[0,1]^d$ and $\nu= \Unif[a,a+1]^d$ with $a\in \RR^d$, \textcite{manole2021sharp} have shown for any $d \ge 1$ and $p>0$ that
\begin{align*}
   \EE\left[T_{\norm{\cdot}^p}(\hat \mu_n, \nu) - T_{\norm{\cdot}^p}(\mu, \nu)\right]
  \gtrsim
  n^{-\frab{p\wedge 2}{d}}.
\end{align*}
This result also remains valid for more generic convex cost functions of the form $c(x, y) = h(x - y)$ under suitable criterions of regularity that depend on $p > 0$ \parencite[Proposition~21]{manole2021sharp}.

More generally, on compact metric spaces $(\XC, d)$ that satisfy the covering condition $\covering(\delta, \XC, d) \asymp 1 + \delta^{-t}$ for some $t>0$ and all $\delta>0$, \textcite{singh2018minimax} have established
\begin{align*}
  \sup_{\mu \in\PC(\XC)} \EE[T_{d^p}(\hat \mu_n, \mu)]
  \gtrsim
  n^{-\frar{1}{2}}+n^{-\frar{p}{t}}
\end{align*}
for all $p\geq 1$ and $n\in \NN$, which implies a convergence rate strictly slower than $n^{-1/2}$ if the metric dimension $t$ is larger than $2p$.
The previous result can also be extended to basic settings affected by lower complexity adaptation, i.e., when the metric space $\X$ has a lower intrinsic dimension than $\Y$.
For example, letting $\Y = \X \times \Z$ for a metric space $(\Z, \tilde{d})$ and considering the $p$-Wasserstein ground costs $d_\times^p = d^p \oplus \tilde{d}^p$ on the product space, it follows by \textcite[Proposition~2.3]{hundrieser2022empirical} that
\begin{align*}
  \sup_{{\mu \in\PC(\XC), \nu \in \PC(\YC)}} \EE\big[T_{d_\times^p}(\hat \mu_n, \nu) - T_{d_\times^p}(\mu, \nu)\big]
  \gtrsim
  n^{-\frar{1}{2}} + n^{-\frar{p}{t}},
\end{align*}
where $\X$ is understood to be embedded into $\Y$ via $x \mapsto (x, z_0)$ for some fixed $z_0\in\Z$.

\subsection{Lower bound induced by tail behavior}\label{subsec:lowerbound:tail}

A comparison of the preceding lower bounds to the upper bounds established on the basis of \Cref{thm:unbounded-rates} in \Cref{subsec:WassersteinEucl} and \ref{subsec:WassersteinPol} shows that our results are essentially sharp (up to occasional $\log$-factors) if sufficiently many moments exist.
If the lack of moments is the limiting factor in \Cref{thm:unbounded-rates}, we below present examples that suggest that the rates we obtain can still be considered to be sharp up to the arbitrarily small term $\epsilon > 0$.

Consider the Euclidean Wasserstein setting with identical measures $\mu = \nu\in\P(\RR^d)$ first.
Assume that $\mu$ has a Lebesgue density proportional to $\norm{x}^{-q-d} 1(\norm{x} \geq 1)$, where $q > p \ge 1$.
Then $\mu$ has finite $r$-th moments $\mint\mu \|\cdot\|^r < \infty$ for any $r\in(0, q)$.
In \textcite{fournier2015rate}, it is shown that $\EE\big[\smash{T_{\norm{\cdot}^p}}(\hat \mu_n, \mu)\big] \gtrsim n^{-\frab{q-p}{q}}$ for any $n\in\NN$.
In the context of \Cref{cor:wasserstein-rates}, this implies the following observation:
for any $q > 1$ and $\epsilon > 0$, there exists some $\mu\in\PC(\RR^d)$ with finite $q$-th Euclidean moment such that
\begin{equation*}
  \EE\,T_{\norm{\cdot}^p}(\hat \mu_n, \mu)
  \gtrsim
  n^{-\frab{q+\epsilon - p}{q+\epsilon}}.
\end{equation*}
Up to the arbitrarily small term $\epsilon > 0$, both in the moment condition and in the rate, this matches the moment-induced part of the upper bound in \Cref{cor:wasserstein-rates}.

To address the general scenario of different population measures $\mu\neq\nu$ on generic ground spaces $\XC$ and $\YC$, we provide another example that admits similar conclusions regarding our moment requirements.
For a lower semi-continuous cost function $c\colon \XC\times \YC \rightarrow \RRplus$, we set $\nu = \delta_{y}$ for some fixed $y\in\Y$.
This simple choice implies that, for any $\mu\in\P(\X)$,
\begin{align*}
  T_c(\mu, \nu)
  =
  \mint\mu c(\cdot, y),
\end{align*}
essentially reducing the empirical estimation of the optimal transport cost to the estimation of the expectation of the random variable $Z = c(X, y)$ with $X\sim\mu$.
To lower-bound the mean absolute deviation when estimating $Z$, we rely on the theory of stable distributions \parencite{nolan2020univariate}.
By Theorem 3.12 in this reference, a measure $\tau\in \PC(\RRplus)$ is in the domain of attraction of an $s$-stable distribution with $s\in (1,2)$, denoted by $\tau \in \DC(s)$, if $\lim_{t\rightarrow \infty} t^s \tau(t,\infty) \in (0,\infty)$ exists.
In this case, there is a non-degenerate random variable $W$ such that i.i.d.\ random variables $Z_1, \dots, Z_n\sim \tau$ satisfy, for $n\rightarrow \infty$,
\begin{equation*}
  n^{\frab{s-1}{s}}\left(\frac{1}{n}\sum_{i = 1}^{n} Z_i - \EE\,Z_1\right)
  \xrightarrow{\;\;\mathcal{D}\;\;}
  W.
\end{equation*}
Note that $\DC(s)$ is non-empty for any $s\in (1,2)$ and that $\EE[Z_1]$ is well-defined.
If the distribution $\tau = c(\cdot, y)_\# \mu$ of $Z = c(X, y)$ is an element of $\DC(s)$, which can be realized by picking a suitable $\mu$ under generic conditions on the costs\footnote{
  Given $s\in(1, 2)$, measures $\mu\in \PC(\XC)$ with $c(\cdot, y)_\#\mu \in \DC(s)$ exist if $c(\cdot,y)$ admits a measurable right-inverse.
  The existence of a Borel measurable right-inverse is a mild assumption and is, e.g., guaranteed for continuous $c$ if there exists a compact set $\XC_k \subseteq \XC$ with $\range\big( c(\cdot,y)|_{\XC_k}\big) = [k-1,k]$ for each $k\in\NN$ \parencite[Theorem~6.9.7]{bogachev2007measure_two}.
}, we can conclude
\begin{equation*}
  \liminf_{n\rightarrow \infty} \EE\left[ n^{\frab{s-1}{s}} \big|T_c(\hat \mu_n, \nu)  - T_c(\mu, \nu)\big| \right]
  \geq
  \EE|W| > 0
\end{equation*}
for a suitable limiting random variable $W$.
This yields, for any $n\in\NN$, that
\begin{equation*}
  \EE\,\big| T_c(\hat \mu_n, \nu)  - T_c(\mu, \nu)\big|
  \gtrsim
  n^{-\frab{s-1}{s}}.
\end{equation*}
Since it can be shown that $\mint\mu c(\cdot, y)^r < \infty$ for any $r\in(1, s)$, we can conclude with the following statement:
for any $s\in(1, 2)$ and $\epsilon > 0$, under mild conditions on $c$, there exist $\mu\in\P(\X)$ and $\nu\in\P(\Y)$ with finite moments $\mint\mu c(\cdot, y)^s$ and $\mint\nu c(x, \cdot)^s$ for some $x\in\X$ and $y\in\Y$, such that
\begin{equation*}
  \EE\,\big|T_c(\hat \mu_n, \hat\nu_n) - T_c(\mu, \nu)\big|
  \gtrsim
  n^{-\frab{s+\epsilon - 1}{s+\epsilon}}.
\end{equation*}
Like in the Euclidean case above, this means that the moment requirements for \Cref{thm:unbounded-rates} are generally sharp up to $\epsilon > 0$.

\section{Proof of the main result}
\label{sec:proof}

The fundamental idea behind the proof of \Cref{thm:unbounded-rates} is to decompose an (arbitrary) optimal transport plan $\pi$ between $\mu$ and $\nu$ into a convex combination
\begin{equation}\label{eq:plan-decomposition}
  \pi = \sum_{l\in\NN} {a_l}\pi_l
\end{equation}
of (optimal) sub-plans $\pi_l$ whose support is restricted to bounded subsets of
$\XY$.
Since the optimal transport plan $\hat\pi_{n,m}$ between the empirical measures $\hat\mu_n$ and $\hat\nu_m$ deviates from the population plan $\pi$, it clearly cannot be decomposed in the exact same way.
Still, we can approximate $\hat\pi_{n,m}$ by a convex combination of empirical sub-plans $\hat\pi_{l,n,m}$ that follows the structure of decomposition \eqref{eq:plan-decomposition}.
Roughly speaking, $\hat\pi_{l,n,m}$ optimally matches as much mass as possible between points that fall into the support of $\pi_l$.
The difference of the empirical sub-plans to the population sub-plans can then be controlled by Assumption $\BC(\kappa, \alpha)$.
The remaining mass, which could not be matched by the plans $\hat\pi_{l,n,m}$, is simply transported via a product coupling.
Since the fraction of non-matching mass vanishes as $n,m\to\infty$, the resulting error terms also vanish in this limit.

In our proof, we construct the sub-plans $\pi_l$ by conditioning $\pi$ onto layered sets $A_l\subset\XY$ defined by the criterion
\begin{equation*}
  2^{l-1} \le \cX \oplus \cY < 2^l,
\end{equation*}
which makes sure that the transport induced by $\pi_l$ is well-behaved and can be controlled for all $l\in\NN$.
In particular, this approach allows us to take care of the problem that mass can in principle be transported arbitrarily far, simply by imposing moment conditions onto the marginals.
No explicit control of the transport plan or map, as it is for example established by \textcite{manole2021sharp} under a set of strong concentration properties, is needed.
Our strategy also circumvents a series of difficulties that arise when trying to approach the convergence of $\T_c(\hat\mu_n, \hat\nu_m)$ in the dual formulation, where analyzing properties of the function class $\F_c$ in unbounded settings can be daunting.
However, for certain well-behaved classes of cost functions, a related dual decomposition approach that we explore in \Cref{app:dual_decomposition} is feasible as well (and even beneficial in some respects).

\subsection{Composition bound and rates}
\label{subsec:composition}

The first step to make our arguments rigorous is established by the following lemma, which provides a general bound for the optimal transport costs between two composite measures in terms of the transport costs between the individual components.

\begin{lemma}{composition bound}{composition-bound}
  Let $\X$ and $\Y$ be Polish, $c\colon\XY \to\RRplus$ measurable,
  $\mu \in \P(\X)$, and $\nu \in \P(\Y)$. Assume $\mu = \sum_{l\in\NN} a_{l}
  \mu_{l}$ with $\mu_{l}\in\P(\X)$ and $\nu = \sum_{l\in\NN} b_{l} \nu_{l}$ with
  $\nu_{l}\in\P(\Y)$ for all $l\in\NN$, where $a,b\in\P(\NN)$.
  If $c_l > 0$ exist such that $\mint{(\mu_l\otimes\nu_k)} c \le c_l + c_k$ for each
  $l,k\in\NN$, then
  \begin{equation}\label{eq:composition-bound}
    T_c(\mu, \nu)
    \le
    \sum_{l\in\NN} \min(a_l, b_l)\, T_c(\mu_l, \nu_l)
    +
    4\sum_{l\in\NN} |a_l - b_l|\,c_l.
  \end{equation}
\end{lemma}

Bounds in the vein of \Cref{lem:composition-bound} have been employed in several previous works on the statistical analysis of the empirical Wasserstein distance on a common Polish metric space $\XC = \YC$, see \textcite[Lemma 2]{dereich2013constructive}, \textcite[Lemma 2.2]{boissard2014mean}, \textcite[Lemma 5]{fournier2015rate}, \textcite[Proposition 1]{weed2019sharp}, \textcite[Lemma 2.2]{lei2020convergence}, and  \textcite[Proposition~4]{fournier2022convergence}.
Recently, a related bound was also derived for costs $c(x,y) = h(\norm{x-y})$ on Euclidean space, where $h\colon \RR_+ \to \RR_+$ grows  polynomially near zero, see \textcite[Inequality 4.1]{larsson2023concentration}.
However, all of these results are tailored to the analysis of $\T_c(\hat\mu_n, \mu)$ under specific norm-related cost structures, and are thus not flexible enough for our needs.
Further, except for \textcite{lei2020convergence}, who formulates a bound for homogenous costs on a Banach~space, and \textcite{larsson2023concentration}, the mentioned bounds do not capture the dependency on the partial transport costs $\T_c(\mu_l, \nu_l)$.
This is a pivotal element for leveraging Assumption $\BC(\kappa, \alpha)$ in the proof of \Cref{thm:unbounded-rates}.
Aside from this, the mentioned approaches are restricted to decompositions $\mu = \sum_{l \in \NN} \mu|_{B_l}$ and $\nu = \sum_{l \in \NN} \nu|_{B_l}$ for a common partition $\XC = \YC = \dot \bigcup_{l\in \NN} B_l$, whereas our bound permits arbitrary partitions of the measures $\mu$ and $\nu$ that can be defined on different spaces.

\begin{proof}[Proof of \Cref{lem:composition-bound}]
Consider the case $\T_c(\mu, \nu) < \infty$ first.
We will construct a transport plan $\pi\in\C(\mu, \nu)$.
Fix $\epsilon > 0$.
For each $l\in\NN$ with $a_l, b_l > 0$, let $\pi_l$ be an $\epsilon$-optimal plan between $\mu_l$ and $\nu_l$ (optimal plans do not have to exist for general measurable cost functions).
Define
  \begin{equation*}
    \bar{\pi}
    =
    \sum_{l\in\NN}\min(a_l, b_l)\,\pi_l.
  \end{equation*}
Then $\bar\pi c \le \sum_{l\in\NN} \min(a_l, b_l)\,T_c(\mu_l, \nu_l) + \epsilon$.
Furthermore, $\bar{\pi}$ has marginals $\bar\mu$ and $\bar\nu$ that satisfy
  \begin{equation*}
    (\mu - \bar\mu) = \sum_{l\in\NN} (a_l - b_l)_+\,\mu_l
    \qquad\text{and}\qquad
    (\nu - \bar\nu) = \sum_{l\in\NN} (b_l - a_l)_+\,\nu_l.
  \end{equation*}
The total mass of $\pi - \bar\pi$ equals
  \begin{equation*}
    (\pi - \bar\pi)(\XY)
    =
    \sum_{l\in\NN} (a_l - b_l)_+
    =
    \sum_{l\in\NN} (b_l - a_l)_+
    =
    \frac{1}{2} \sum_{l\in\NN} |a_l - b_l|
    \eqqcolon
    \frac{m}{2}.
  \end{equation*}
Defining $\tilde\pi = \frac{2}{m}(\mu - \bar\mu)\otimes(\nu - \bar\nu)$, one can check that $\pi = \bar\pi + \tilde\pi \in \C(\mu, \nu)$ is a feasible transport plan.
In particular, we find
  \begin{align*}
    \tilde\pi c
    &=
    \frac{2}{m}\int c(x, y) \dif (\mu - \bar\mu)(x) \dif(\nu - \bar\nu)(y) \\
    &\le
    \frac{2}{m}\sum_{l, k\in\NN} (a_l - b_l)_+(b_k - a_k)_+ \, (c_l + c_k) \\
    &\le
    \frac{2}{m}\sum_{l, k\in\NN} |a_l - b_l||a_k - b_k|\,(c_l + c_k) \\
    &\le
    4\sum_{l\in\NN} |a_l - b_l|\,c_l.
  \end{align*}
Since $\epsilon$ is arbitrary, noting that $T_c(\mu, \nu) \le \pi c$ shows the claimed inequality of the lemma.
In case that $\T_c(\mu, \nu) = \infty$, it suffices to observe that $\pi$ as constructed above is still a transport plan between $\mu$ and $\nu$ for arbitrary $\pi_l\in\CC(\mu_l, \nu_l)$, so the right-hand side of \eqref{eq:composition-bound} is infinite as well.
\end{proof}

For the next result, we operate on a given decomposition \eqref{eq:plan-decomposition} for an optimal plan $\pi$ and apply \Cref{lem:composition-bound} in order to control how much $\T_c(\hat\mu_n, \hat\nu_m)$ can exceed $\T_c(\mu, \nu)$.
The other direction, i.e., how much $\T_c(\mu, \nu)$ can exceed $\T_c(\hat\mu_n, \hat\nu_m)$, is handled by an argument stemming from the dual formulation \eqref{eq:dual} of optimal transport.
To prevent needless confusion, we stress that the sequence $(b_l)_{l\in\NN}$ introduced in the theorem below is merely a technical device that facilitates the ensuing argumentation.
Conceptually, $b_l$ can be identified with the probabilities $a_l$.
We employ the notation $\hat\mu_{l,n}$ and $\hat\nu_{l, m}$ to denote empirical measures of $\mu_l \in \P(\X)$ and $\nu_l\in\P(\Y)$ with $n$ respectively $m$ samples.

\begin{theorem}{composition rates}{composition-rates}
  Let $\X$ and $\Y$ be Polish, $c\colon\XY \to\RRplus$ lower semi-continuous, and $\pi\in\C(\mu, \nu)$ optimal for $\mu \in \P(\X)$ and $\nu\in\P(\Y)$.
  Assume $\pi = \sum_{l\in\NN} a_l\,\pi_l$, where $a \in \P(\NN)$ and $\pi_l\in\C(\mu_l, \nu_l)$ with $\mu_l\in\P(\X)$ and $\nu_l\in\P(\Y)$ for all $l\in\NN$.
  Further, assume there are $c_l > 0$ with $c \le c_l + c_k$ on the support of $\mu_l\otimes\nu_k$ for all $l, k\in\NN$, and $b_l \ge a_l$ such that $(b_l)_{l\in\NN}$ is monotonically decreasing and admits the existence of $l_n\in\NN$ with $b_{l_n} \asymp 1/n$.
  If there are $\alpha,\beta,\gamma\in(0, 1/2]$ and $r_l > 0$ with
  \begin{equation}\label{eq:bounded_rates_composition_rate}
    \Exp\,\big|\T_c(\hat\mu_{l,n}, \hat\nu_{l,m}) - \T_c(\mu_l, \nu_l)\big|
    \le
    r_l \big(n^{-\alpha} + m^{-\alpha}\big)
  \end{equation}
  for all $l, n, m\in\NN$, $\sum_{l\in\NN} r_l b_l^{1-\beta} < \infty$ and $\sum_{l\in\NN} c_l b_l^{1-\gamma} < \infty$, then, for all $n,m\in\NN$,
  \begin{equation}\label{eq:upperbound_composition_rate}
    \Exp\,\big|\T_c(\hat\mu_n, \hat\nu_m) - \T_c(\mu, \nu)\big|
    \lesssim
    n^{-\min(\alpha, \beta, \gamma)} + m^{-\min(\alpha, \beta, \gamma)}.
  \end{equation}
\end{theorem}

\begin{remark}{constants}{}
  A look at the proof of \Cref{thm:composition-rates} reveals how the underlying constant in inequality \eqref{eq:upperbound_composition_rate} may be controlled.
  First, it depends on the sums $\smash{\sum_{l\in\NN} r_l b_l^{1-\beta}}$ and $\smash{\sum_{l\in\NN} c_l \, b_l^{1-\gamma}}$, the latter of which also upper bounds the integral $\mint{(\mu\otimes\nu)} c^p$ for $p = 1/(1-\gamma)$. 
  Secondly, it depends on the $p$-th moments of (arbitrary) dual optimizers $f\in L^p(\mu)$ and $g\in L^p(\nu)$ with $\T_c(\mu, \nu) = \mint\mu f + \mint\nu g$, which are shown to exist in the proof.
  Thirdly, it depends on the value $\rho\geq 1$ that fulfills the property $\frac{1}{\rho n}\leq b_{l_n} \leq \frac{\rho}{n}$ for all $n\in \NN$, which has to exist by assumption.
\end{remark}

\begin{proof}
  Let $p = 1/(1-\gamma) \in (1, 2]$.
  In a first step, we observe
  \begin{equation*}
    \mint{(\mu\otimes\nu)} c^p
    =
    \sum_{l,k\in\NN} a_l a_k (\mu_l\otimes\nu_k) c^p
    \le
    2^p \sum_{l,k\in\NN} \big(c_l^p + c_k^p\big)\,a_l a_k
    \le
    8 \sum_{l\in\NN} c_l^p b_l
    \le
    C
    \sum_{l\in\NN} c_l b_l^{1-\gamma}
    <
    \infty,
  \end{equation*}
  where the existence of a suitable $C > 0$ follows from $c_l^p b_l < c_lb_l^{1-\gamma}$ for large $l$, which is implied by $\smash{\sum_{l\in\NN} c_l b_l^{1-\gamma} < \infty}$.
  According to \Cref{lem:lp-potentials}, there are thus functions $f\in L^p(\mu)$ and $g\in L^p(\nu)$ such that $\T_c(\mu, \nu) = \mu f + \nu g$.
  Hence, it follows for empirical measures $\hat \mu_n, \hat \nu_m$ that $f\in L^1(\hat \mu_n)$ and $g\in L^1(\hat \nu_m)$ almost surely.
  The dual formulation \eqref{eq:dual} of optimal transport thus ensures $T_c(\hat\mu_n, \hat\nu_m) \ge \mint{\hat\mu_n} f + \mint{\hat\nu_m} g$ (almost surely), from which we derive the lower bound
  \begin{equation*}
    (\hat\mu_n - \mu) f + (\hat\nu_m - \nu) g
    \le
    \T_c(\hat\mu_n, \hat\nu_n) - \T_c(\mu, \nu).
  \end{equation*}
  
  For an upper bound, we note that the optimality of $\pi = \sum_{l\in\NN} a_l\pi_l$ implies $\pi_l\in\C(\mu_l, \nu_l)$ to be optimal as well \parencite[Theorem 5.19]{villani2008optimal}.
  Thus, $T_c(\mu, \nu) = \sum_{l\in\NN} a_l T_c(\mu_l, \nu_l)$.
  Next, since we are only interested in distributional properties, we can assume that the empirical samples generating $\hat\mu_n$ and $\hat\nu_m$ are obtained in a two-stage procedure:
  first, a multinomial variable $N = (N_l)_{l\in\NN} \sim \Mult(n, a)$ is drawn.
  Then, $N_l$ points $X_{l,1}, \ldots, X_{l, N_{l}}$ are i.i.d.-sampled from $\mu_{l}$.
  Likewise, we treat the sampling of $Y_1, \ldots, Y_m$ with multinomial frequencies $M \sim \Mult(m, a)$.
  Let
  \begin{equation*}
    \tilde\mu_{l, n}
    = 
    \frac{1}{N_l}\sum_{i = 1}^{N_l} \delta_{X_{l, i}}
    \qquad\text{and}\qquad
    \tilde\nu_{l, m}
    =
    \frac{1}{M_l}\sum_{i = 1}^{M_l} \delta_{Y_{l, i}}
  \end{equation*}
  whenever $N_l > 0$ or $M_l > 0$, respectively.
  It follows by construction that $ \hat\mu_n = \sum_{l\in\NN} \frac{N_l}{n}\,\tilde\mu_{l,n}$ and likewise for $\hat\nu_m$.
  Crucially, conditional on $N_l = n_l$ and $M_l = m_l$, the measure $\tilde\mu_{l, n}$ is distributed like an empirical measure of $\mu_l$ with $n_l$ samples, and $\tilde\nu_{l, m}$ like an empirical measure of $\nu_l$ with $m_l$ samples.
  Observing $\mint{(\tilde\mu_{l, n}\otimes \tilde\nu_{k, m})}c \le c_l + c_k$ by assumption, we can make use of \Cref{lem:composition-bound} to obtain
  \begin{align*}
    T_c(\hat\mu_n, \hat\nu_m)
    \le
    \sum_{l\in\NN} \left(\frac{N_l}{n}\wedge \frac{M_l}{m}\right)\, T_c(\tilde\mu_{l, n}, \tilde\nu_{l, m})
    +
    4\,\sum_{l\in\NN} c_l \left|\frac{N_l}{n} - \frac{M_l}{m}\right|,
  \end{align*}
  where we set $T_c(\tilde\mu_{l, n}, \tilde\nu_{l, m}) = 0$ whenever $\tilde\mu_{l, n}$ or $\tilde\nu_{l, m}$ are not well-defined (i.e., when $N_l M_l = 0$).
  Furthermore, we notice $\T_c(\tilde\mu_{l, n}, \tilde\nu_{l, m}) \le 2c_l$ and thus (generously) estimate
  \begin{align*}
   \left(\frac{N_l}{n}\wedge \frac{M_l}{m}\right)\, \T_c(\tilde\mu_{l, n}, \tilde\nu_{l, m}) - a_l \T_c(\mu_l, \nu_l)
    &\le
    a_l \big(\T_c(\tilde\mu_{l, n}, \tilde\nu_{l, m}) - \T_c(\mu_l, \nu_l)\big)
    + 2c_l \left|\left(\frac{N_l}{n}\wedge \frac{M_l}{m}\right) - a_l\right| \\
    &\le
    1(N_lM_l > 0)\,a_l \big(\T_c(\tilde\mu_{l, n}, \tilde\nu_{l, m}) - \T_c(\mu_l, \nu_l)\big) \\
    &\qquad+ \frac{2}{n}c_l\,|N_l - na_l| + \frac{2}{m}c_l\,|M_l - ma_l|,
  \end{align*}
  where we have inserted a non-trivial zero to establish the first inequality, and used that
  \begin{equation*}
    a - c
    \leq
    \min(a,b) - c
    \leq
    b - c
    \quad ~\text{for}~ a,b,c\in \RR, a \leq b
  \end{equation*}
  to derive the second inequality.
  Consequently,
  \begin{align*}
    T_c(\hat\mu_n, \hat\nu_m) - T_c(\mu, \nu)
    &\le
    \sum_{l\in\NN} \left(\left(\frac{N_l}{n}\wedge \frac{M_l}{m}\right)\, T_c(\tilde\mu_{l, n}, \tilde\nu_{l, m})
    - a_l T_c(\mu_l, \nu_l)\right)
    +
    4\,\sum_{l\in\NN} c_l \left|\frac{N_l}{n} - \frac{M_l}{m}\right|\\
    &\le
    \sum_{l\in\NN} 1(N_lM_l > 0)\, a_l \left(T_c(\tilde\mu_{l, n}, \tilde\nu_{l, m})
    - T_c(\mu_l, \nu_l)\right) \\
    &\qquad+
    \frac{6}{n}\sum_{l\in\NN} c_l \left|N_l - na_l\right| + \frac{6}{m}\sum_{l\in\NN}c_l \left|M_l - ma_l\right|.
  \end{align*}
  Due to the independence of the samples, the measures $\tilde\mu_{l, n}$ and $\tilde\nu_{l, m}$ are empirical measures of $\mu_l$ and $\nu_l$ when conditioned on $N_l$ and $M_l$.
  By assumption \eqref{eq:bounded_rates_composition_rate}, we hence find
  \begin{equation}\label{eq:upperbound_independence}
    \Exp\big[\T_c(\tilde\mu_{l, n}, \tilde\nu_{l, m}) - \T_c(\mu_l, \nu_l)\,\big|\, N_l, M_l\big]
  \le
  r_l \big(N_l^{-\alpha} + M_l^{-\alpha}\big)
  \end{equation}
  for $N_lM_l > 0$.
  Putting together the established upper and lower bounds, we obtain
  \begin{align}
    \Exp\,\big|\T(\hat\mu_n, \hat\nu_m) - \T(\mu, \nu)\big|
    &\le
    \sum_{l\in\NN} r_l a_l\,\Exp\big[1(N_l > 0) N_l^{-\alpha} + 1(M_l > 0)M_l^{-\alpha}\big] \nonumber \\
    &\qquad+
    \frac{6}{n}\sum_{l\in\NN} c_l\,\Exp\,\big|N_l - na_l\big|
    + \frac{6}{m}\sum_{l\in\NN} c_l\,\Exp\,\big|M_l - ma_l\big| \label{eq:composite-rates-components} \\
    &\qquad+
    \Exp\big[|(\hat\mu_n - \mu) f|\big] + \Exp\big[|(\hat\nu_m - \nu) g|\big]. \nonumber
  \end{align}
  Since $f$ and $g$ are $p$-integrable with respect to $\mu$ and $\nu$, the two final terms have $n^{-(p-1)/p} = n^{-\gamma}$ and $m^{-\gamma}$ rates with underlying constants that depend on the magnitude of the $p$-th moments (see \Cref{lem:sample-mean}).
  Furthermore, we apply \Cref{lem:binomial-sum-convergence} to establish
  \begin{equation*}
    \frac{1}{n}\sum_{l\in\NN} c_l\,\Exp\,\big|N_l - na_l\big|
    \lesssim
    n^{-\gamma}
  \end{equation*}
  with an underlying constant that is determined by $\sum_{l\in \NN} c_l b_l^{1- \gamma}$ and the value $\rho\geq 1$ such that $\frac{1}{\rho n} \leq b_{l_n} \leq \frac{\rho}{n}$ for all $n\in \NN$.
  An analogous inequality holds for $M_l$.
  Finally, \Cref{lem:inverse-binomial-moment} with the choices $\varphi(x) = x^\alpha$ and $a = 2$ implies $\smash{\Exp\big[1(N_l > 0)N_l^{-\alpha}\big]} \le 2(na_l)^{-\alpha}$.
  If $\alpha \le \beta$, we simply conclude
  \begin{subequations}%
  \begin{equation}\label{eq:bounded-rate-binomial-1}
    \sum_{l\in\NN} r_la_l\,\Exp\big[1(N_l > 0)N_l^{-\alpha}\big]
    \le
    2n^{-\alpha} \sum_{l\in\NN} r_la_l^{1-\alpha}
    \lesssim
    n^{-\alpha},
  \end{equation}
  where we have used $a_l^{1-\alpha} \le a_l^{1-\beta} \le b_l^{1-\beta}$ to control the sum on the right.
  If $\beta < \alpha$, on the other hand, we employ \Cref{lem:inverse-binomial-moment} with $\varphi(x) = \smash{x^\beta}$ to
  derive
  \begin{equation}\label{eq:bounded-rate-binomial-2}
    \sum_{l\in\NN} r_la_l\,\Exp\big[1(N_l > 0)N_l^{-\alpha}\big]
    \le
    \sum_{l\in\NN} r_la_l\,\Exp\big[1(N_l > 0)N_l^{-\beta}\big]
    \le
    2n^{-\beta} \sum_{l\in\NN} r_la_l^{1-\beta}
    \lesssim
    n^{-\beta}.
  \end{equation}
  \label{eq:bounded-rate-binomial}%
  \end{subequations}%
  Since these arguments hold for $M_l$ as well, the claim of the theorem is established.
\end{proof}

From here on, it is straightforward to establish \Cref{thm:unbounded-rates} as an application of \Cref{thm:composition-rates}.
We mainly have to construct a suitable decomposition of an optimal transport plan $\pi\in\CC(\mu, \nu)$ and verify the moment conditions required by \Cref{thm:composition-rates} based on the assumptions $\mint\mu\cX^{s+\epsilon} < \infty$ and $\mint\nu\cY^{s+\epsilon} < \infty$.

\begin{proof}[Proof of \Cref{thm:unbounded-rates}]
  For convenience, we will assume in the following that $\cX, \cY \ge 1$.
  This is always possible by offsetting the marginal costs $\cX$ and $\cY$, which (a) does not affect the moment conditions required in the theorem and (b) leaves Assumption~$\BC(\kappa, \alpha)$ unimpaired.

  Set $A_l = \big\{(x,y)\in\XY\,|\, 2^l \le \cX(x) + \cY(y) < 2^{l+1}\big\}$ for $l\in\NN$.
  Let $\pi\in\CC(\mu, \nu)$ be an optimal transport plan and set $a_l = \pi(A_l)$ as well as $\pi_l = \pi|_{A_l}/a_l$ (which is well-defined if $a_l > 0$). Then,
  \begin{equation*}
    \pi = \sum_{l\in\NN} a_l \pi_l.
  \end{equation*}
  Denote the marginals of $\pi_l$ as $\mu_l$ and $\nu_l$.
  Letting $c_l = 2^{l+1}$, note that the support of $\mu_l$ is contained in $B_\X(c_l)$ and the one of $\nu_l$ in $B_\Y(c_l)$.
  Therefore, $c \le \cX \oplus \cY \le c_l + c_k$ holds on the support of
  $\mu_l\otimes\nu_k$ for any $l,k\in\NN$.
  Due to Assumption~$\BC(\kappa, \alpha)$, the choice $r_l = \kappa 2^{l+1}$ yields
  \begin{equation}\label{eq:aux-bounded-convergence}
    \Exp\,\big|\T_c(\hat\mu_{l,n}, \hat\nu_{l,m}) - \T_c(\mu_l, \nu_l)\big|
    \le
    r_l \big(n^{-\alpha} + m^{-\alpha}\big)
  \end{equation}
  for all $l, n, m\in\NN$.
  Next, we use the property $\cX \oplus \cY \ge 2^l$ on $A_l$ to establish
  \begin{align*}
    \sum_{l\in\NN} c_l^{s+\epsilon} a_l
    =
    2^{s+\epsilon} \sum_{l\in\NN} 2^{(s+\epsilon)l} \pi(A_l)
    \le
    2^{s+\epsilon}\,\pi(\cX \oplus \cY)^{s+\epsilon}
    \le
    4^{s+\epsilon} \big(\mint\mu \cX^{s+\epsilon} + \mint\nu \cY^{s+\epsilon}\big)
    <
    \infty,
  \end{align*}
  where the moment assumptions placed on $\mu$ and $\nu$ have been exploited.
  In particular, we obtain the insight that $a_l \le b_l \coloneq K 2^{-l(s+\epsilon)}$ for $K = 4^{s+\epsilon} \big(\mint\mu \cX^{s+\epsilon} + \mint\nu \cY^{s+\epsilon}\big)$.
  We now set $\beta = \gamma = (s - 1)/s$.
  Observing $(s+\epsilon)(1-\beta) = 1 + \epsilon/s$, we conclude
  \begin{subequations}
  \begin{align}
    \sum_{l\in\NN} r_l b_l^{1-\beta}
    =
    2\kappa\,K \sum_{l\in\NN} 2^{l} 2^{-l(s+\epsilon)(1-\beta)}
    =
    \frac{2\kappa K}{1 - 2^{-\epsilon/s}}
    <
    \infty
  \shortintertext{and similarly}
    \sum_{l\in\NN} c_l b_l^{1-\gamma}
    =
    2K \sum_{l\in\NN} 2^l 2^{-l(s+\epsilon)(1-\gamma)}
    =
    \frac{2K}{1 - 2^{-\epsilon/s}}
    <
    \infty.
  \end{align}
  \label{eq:bounded-rates-constants}%
  \end{subequations}
 Since $b_l$ is monotonically decreasing and the choice $l_n = \big\lceil\log_2(Kn)/(s+\epsilon)\big\rceil$ satisfies $1/2n \le b_{l_n} \le 1/n$, all conditions for \Cref{thm:composition-rates} are satisfied. 
\end{proof}

Finally, we substantiate the claims that follow \Cref{thm:unbounded-rates} in the introduction.
Points \ref{it:mu-equals-nu} to \ref{it:independence} are direct conclusions of the assumptions or the proof and require no further explanation.
In contrast, points \ref{it:log}, \ref{it:one-sample}, and \ref{it:uniform} are more subtle and deserve additional justification.

\begin{itemize}[leftmargin=2em]
  \item[\ref{it:log}] In order to cover settings with additional log-terms in Assumption~$\BC(\kappa, \alpha)$, for example $n^{-\alpha}\log(n + 1)^\delta$ for $\delta > 0$ instead of just $n^{-\alpha}$, we have to reconsider inequalities \eqref{eq:bounded-rate-binomial-1} and \eqref{eq:bounded-rate-binomial-2} in the proof of \Cref{thm:composition-rates}.
    In case that $\alpha \le \beta$, we observe that
    \begin{equation*}
      \Exp\,\big[1(N_l > 0)N_l^{-\alpha}\log(N_l + 1)^\delta\big]
      \le
      \Exp\,\big[1(N_l > 0)N_l^{-\alpha}\big] \log(n + 1)^\delta
    \end{equation*}
    for all $l\in\NN$. Thus, the term $\log(n+1)^\delta$ appears as an additional
    factor on the right-hand side of \eqref{eq:bounded-rate-binomial-1}.
    If $\beta < \alpha$, on the other hand, we note that there exists $C > 0$ such that $\log(k + 1)^\delta \le C \, k^{\alpha-\beta}$ for all $k\in\NN$ and conclude
    \begin{equation*}
      \Exp\,\big[1(N_l > 0)N_l^{-\alpha}\log(N_l + 1)^\delta\big]
      \le
      C\,\Exp\,\big[1(N_l > 0)N_l^{-\beta}\big],
    \end{equation*}
    which shows that the rate in \eqref{eq:bounded-rate-binomial-2} stays untouched (albeit with a different constant that depends on $\alpha - \beta$ and $\delta$).
    We conclude that the term $\log(n+1)^\delta$ is replicated in the upper bound \eqref{eq:unbounded-rates} of \Cref{thm:unbounded-rates} in case that $\alpha \le s/(1+s)$, but that there is no difference in the result if $s/(1+s) < \alpha$.
    The same argumentation can be used to treat generic rates $n^{-\alpha} h(n)$, where $h$ is monotonically increasing and grows slower than any power of $n$.

  \item[\ref{it:one-sample}] Employing the dual formulation \eqref{eq:dual} of optimal transport, we observe for any $m\in\NN$ that
    \begin{align*}
      \Exp\,\big|\T_c(\hat\mu_{n}, \nu) - \T_c(\mu, \nu)\big|
      &=
      \Exp\,\big|\max_{f,g}\big(\hat\mu_{n} f + \mint\nu g\big) - \T_c(\mu, \nu)\big| \\
      &=
      \Exp\,\big|\max_{f,g}\big(\hat\mu_{n} f + \Exp[\hat\nu_m g]\big) - \T_c(\mu, \nu)\big| \\
      &\le
      \Exp\,\big|\max_{f,g}\big(\hat\mu_{n} f + \hat\nu_m g\big) - \T_c(\mu, \nu)\big| \\
      &=
      \Exp\,\big|\T_c(\hat\mu_{n}, \hat\nu_m) - \T_c(\mu, \nu)\big|.
    \end{align*}
    In the third line, we utilized independence of $\hat \mu_n$ and $\hat \nu_m$, and that pulling the expectation out of the supremum and the absolute value only increases the value via Jensen's inequality.

  \item[\ref{it:uniform}]
    In order to bound the constants in \Cref{thm:unbounded-rates}, we have to understand the constants of the different contributions on the right-hand side of inequality \eqref{eq:composite-rates-components} in the proof of \Cref{thm:composition-rates}.
    For the first line, equation~\eqref{eq:bounded-rate-binomial} implies constants that are a universal multiple of $R = \smash{\sum_{l\in\NN} r_lb_l^{1-\beta}}$.
    The constants in the second line of \eqref{eq:composite-rates-components} are
    documented by \Cref{lem:binomial-sum-convergence}, where we note that $1/2n \le b_{l_n} \le 2/n$ (i.e., $\rho = 2$) for the choice of $b_l$ and $l_n$ in the proof of \Cref{thm:unbounded-rates}.
    Therefore, the respective constants are bounded by multiples of $C = \sum_{l\in\NN} c_lb_l^{1-\gamma}$.
    According to \eqref{eq:bounded-rates-constants}, the values of $R$ and $C$ are each bounded by
    \begin{equation*}
      \frac{2\cdot 4^{s+\epsilon} (\kappa + 1)}{1 - 2^{-\epsilon/s}}
      \big(\mint\mu\cX^{s+\epsilon} + \mint\nu\cY^{s+\epsilon}\big).
    \end{equation*}
    Finally, to handle the third line of inequality \eqref{eq:composite-rates-components}, we note that the $s$-th moments of the dual solutions $f$ and $g$ are bounded by multiples of $\mint\mu\cX^{s+\epsilon} + \mint\nu\cY^{s+\epsilon}$ as well (see the second part of \Cref{lem:lp-potentials}).
    In concert with \Cref{lem:sample-mean}, this shows that the constants of all terms occurring in \eqref{eq:composite-rates-components} rely on $\kappa$, $s$, $\epsilon$, $\alpha$, as well as the $(s+\epsilon)$-th cost moments only.
\end{itemize}

\subsection{Auxiliary results}

This subsection contains technical auxiliary statements needed in the proof of \Cref{thm:composition-rates}.
We show that the dual solutions have finite $p$-th moments if the costs do (\Cref{lem:lp-potentials}), that the sample mean of a random variable with finite $p$-th moment for $p\in(1, 2]$ converges with rate $n^{-(p-1)/p}$ (\Cref{lem:sample-mean}), and that the expectation of certain binomial expressions can be bounded suitably (\Cref{lem:binomial-sum-convergence} and \ref{lem:inverse-binomial-moment}).

\begin{lemma}{integrable dual solutions}{lp-potentials}
  Let $\X$ and $\Y$ be Polish, $c\colon\X\times\Y \to \RRplus$ lower semi-continuous, $\mu\in\P(\X)$, and $\nu\in\P(\Y)$.
  If $(\mu\otimes\nu) c^p < \infty$ for some $p \ge 1$, there exist functions $f\in L^p(\mu)$ and $g\in L^p(\nu)$ such that $f\oplus g \le c$ and $\T_c(\mu, \nu) = \mint\mu f + \mint\nu g$.
  \\[0.5em]
  If furthermore $c \le \cXY$ for measurable functions $\cX\colon \X \to \RRplus$ and $\cX\colon \Y \to \RRplus$, we may assume that
  \begin{equation*}
    \mint\mu |f|^p + \mint\nu |g|^p
    \le
    8^{p+1}\,\big(\mint\mu\cX^p + \mint\nu\cY^p\big).
  \end{equation*}
\end{lemma}

\begin{proof}
  According to Theorem~5.10.ii together with Remark~5.14 in \textcite{villani2008optimal}, there exist dual solutions $f\in L^1(\mu)$ and $g\in L^1(\nu)$.
  To show that they are in $L^p$, fix $(x_0, y_0)\in\XY$ such that $f(x_0)$, $g(y_0)$, $\nu c(x_0, \cdot)^p$, and $\mu c(\cdot, y_0)^p$ are all finite.
  This is possible since the points $(x_0, y_0)$ with these properties must have full $\mu\otimes\nu$-mass due to integrability of $f$, $g$, and $c^p$.
  Then, observe that $f(x) \le c(x, y_0) - g(y_0)$ as well as $g(y) \le c(x_0, y) - f(x_0)$ for all $(x,y)\in\XY$ due to the duality constraint $f\oplus g \le c$.
  Hence, for any $(x,y)\in\XY$ with $f(x) + g(y) = c(x,y)$, it holds that
  \begin{equation*}
    f(x)
    \le
    c(x, y_0) - g(y_0)
    \qquad\text{and}\qquad
    f(x) =
    c(x, y) - g(y)
    \ge
    - g(y)
    \ge
    -c(x_0, y) + f(x_0),
  \end{equation*}
  which in turn implies
  \begin{align}
    |f(x)|^p
    &\le
    \big(c(x_0, y) + c(x, y_0) + |f(x_0)| + |g(y_0)|\big)^p \nonumber\\
    &\le
    4^p \big(c(x_0, y)^p + c(x, y_0)^p + |f(x_0)|^p + |g(y_0)|^p\big).
    \label{eq:dual-p-bound}
  \end{align}
  Since the set $S = \big\{(x,y)\in\XY\,|\,f(x) + g(y) = c(x,y)\big\}$ has full $\pi$-mass, where $\pi\in\CC(\mu, \nu)$ is an optimal transport plan, integrating this inequality over $\pi$ shows $\mint\mu |f|^p < \infty$.
  The bound $\mint\nu |g|^p < \infty$ follows analogously.

  For the second claim, let $M = \smash{\mint\mu\cX^p + \mint\nu\cY^p}
  = \mint\pi\smash{\big(\cX^p\oplus\cY^p\big)}$.
  If $M = \infty$, the inequality is trivially true, so we can assume $M < \infty$. 
  Then, there is $(x_0, y_0) \in \supp\,\pi$ with $\cX^p(x_0) + \cY^p(y_0) \le M$ and $f(x_0) + g(y_0) = c(x_0, y_0)$.
  In particular, $f(x_0)$ and $g(y_0)$ are finite.
  Due to the shift-invariance of dual solutions, we may assume $f(x_0) = g(y_0) = c(x_0, y_0) / 2$, which implies
  \begin{equation*}
    |f(x_0)|^p + |g(y_0)|^p
    \le
    2 \left(\frac{\cX(x_0) + \cY(y_0)}{2}\right)^p
    \le
    \cX(x_0)^p + \cY(y_0)^p.
  \end{equation*}
  From inequality \eqref{eq:dual-p-bound} we deduce for any $(x, y)\in S$ that
  \begin{align*}
    |f(x)|^p
    &\le
    4^p \big(c(x_0, y)^p + c(x, y_0)^p + |f(x_0)|^p + |g(y_0)|^p\big) \\
    &\le
    2\cdot 8^p \big(\cX(x_0)^p + \cY(y_0)^p + \cY(y)^p + \cX(x)^p\big).
  \end{align*}
  Integration over $\pi$ shows $\mint\mu |f|^p \le 4\cdot 8^p M$.
  Since the roles of $f$ and $g$ in \eqref{eq:dual-p-bound} can be reversed, the claim is established.
\end{proof}

\begin{remark}{integrability for $p < 1$}{}
  The argumentation in the proof of \Cref{lem:lp-potentials} also works for $0 < p < 1$ under the additional assumption $T_c(\mu, \nu) < \infty$.
  Then, generalized dual solutions $f\in L^p(\mu)$ and $g\in L^p(\nu)$ exist, for which equality in $f \oplus g \le c$ holds $\pi$-almost surely for every optimal $\pi\in\C(\mu, \nu)$ (see \cite[Theorem~5.10.ii]{villani2008optimal}), but where equality in $T_c(\mu, \nu) = \mu f + \nu g$ may fail to hold due to lacking integrability.
\end{remark}

\vspace{-1em}

\begin{remark}{moment conditions}{sharp-moments}
The $L^p$-integrability of dual solutions in \Cref{lem:lp-potentials} crucially relies on the condition $(\mu\otimes \nu)c^p<\infty$.
In fact, even if only one of the measures $\mu$ or $\nu$ does not have finite moments with respect to its marginal costs, e.g., $\mint\mu\cX^p < \infty$ and $\mint\nu\cY^p = \infty$, then \emph{all} dual solutions $f$ and $g$ may already fail to be $L^p$-integrable.
An explicit example that documents this behavior on the real line is provided in \Cref{app:non-integrable-potentials}.
Consequently, we infer that $L^p$-integrability of dual solutions generally requires sufficient concentration of \emph{both} underlying measures, which asserts the conditions of \Cref{lem:lp-potentials} to be essentially sharp.    
\end{remark}

\begin{lemma}{sample mean convergence rates}{sample-mean}
  Let $X, X_1, \ldots, X_n$ for $n\in\NN$ be i.i.d.\ real valued random variables and let $1 < p \le 2$.
  Then
  \begin{equation*}
    \Exp\left|\frac{1}{n}\sum_{i=1}^n X_i - \Exp\,X\right|
    \le
    \left(2\,\Exp\,\big|X - \Exp\,X\big|^p\right)^{1/p}\,
    n^{-\frab{p-1}{p}}.
  \end{equation*}
\end{lemma}
\begin{proof}
  We can assume $\Exp\,\big|X - \Exp\,X\big|^p < \infty$, otherwise the claim is trivial.
  Let $S_n = \sum_{i=1}^n \big(X_i - \Exp\,X\big)$.
  We first employ Jensen's inequality and then Theorem~2 from \textcite{bahr1965inequalities} to derive
  \begin{equation*}
    \frac{1}{n}\,\Exp\,|S_n|
    \le
    \frac{1}{n}\big(\Exp\,|S_n|^p\big)^{1/p}
    \le
    \frac{1}{n} \left(2\sum_{i=1}^n \Exp\,\big|X_i - \Exp\,X\big|^p\right)^{1/p}
    \le
    \left(2\,\Exp\,\big|X - \Exp\,X\big|^p\right)^{1/p}\,
    n^{-\frab{p-1}{p}}. \qedhere
  \end{equation*}
\end{proof}

\begin{lemma}{multinomial convergence}{binomial-sum-convergence}
  Consider $N = (N_l)_{l\in\NN} \sim \Mult\big(n, (a_l)_{l\in\NN}\big)$ for $a\in\P(\NN)$ and $n\in\NN$.
  Let $b_l, c_l > 0$ and assume that $b_l$ is decreasing in $l$ and satisfies $a_l \le b_l$ for all $l\in\NN$.
  Further, let $\gamma \in (0, 1/2]$.
  If there exist $\rho \ge 1$ and $l_n\in\NN$ such that $\frac{1}{\rho n} \le b_{l_n} \le \frac{\rho}{n}$ for all $n\in\NN$, then
  \begin{equation*}
    \sum_{l\in\NN} c_l\, \Exp\left| \frac{N_l}{n} - a_l\right|
    \le
    3\sqrt{\rho} \left(\sum_{l\in\NN} c_l \, \smash{b_l^{1-\gamma}}\right)
    n^{-\gamma}.
  \end{equation*}
\end{lemma}

\begin{proof}
  First, we observe $\Exp\,|N_l / n - a_l| \le \big(\Exp\,|N_l / n - a_l|^2\big)^{1/2} \le \sqrt{b_l/n}$, where we used that the variance of $N_l\sim\Bin(n, a_l)$ is $na_l (1-a_l)$.
  Next, realizing that $\smash{b_l^{\gamma - 1/2}}$ is non-decreasing due to $\gamma \le 1/2$, we derive
  \begin{equation*}
    \sum_{l=1}^{l_n} c_l\,\Exp\left|\frac{N_l}{n} - a_l\right|
    \le
    n^{-1/2} \sum_{l=1}^{l_n} c_l b_l^{1-\gamma} b_l^{\gamma - 1/2}
    \le
    n^{-1/2} b_{l_n}^{\gamma - 1/2} \sum_{l=1}^{l_n} c_l b_l^{1-\gamma}
    \le
    \sqrt\rho \left(\sum_{l=1}^{\infty} c_l b_l^{1-\gamma}\right) n^{-\gamma}.
  \end{equation*}
  For the second part of the sum, we use that $b_l$ monotonically decays and bound
  \begin{equation*}
    \sum_{l=l_n}^\infty c_l\,\Exp\left|\frac{N_l}{n} - a_l\right|
    \le
    \sum_{l=l_n}^\infty c_l \left(\Exp\left[\frac{N_l}{n}\right] + a_l\right)
    \le 
    2b_{l_n}^{\gamma}\sum_{l=l_n}^\infty c_l b_l^{1-\gamma}
    \le
    2\sqrt\rho \left(\sum_{l=1}^\infty c_l b_l^{1-\gamma}\right)
    n^{-\gamma}. \qedhere
  \end{equation*}
\end{proof}

\begin{lemma}{inverse truncated binomial moments}{inverse-binomial-moment}
  Let $N \sim \Bin(n, p)$ for $n\in\NN$ and $p\in(0, 1)$.
  Let $\varphi\colon(0,1]\to[0, \infty]$ be increasing and concave with $\varphi(k^{-1}) \le a\,\varphi\big((k+1)^{-1}\big)$ for some $a > 0$ and all $k\in\NN$.
  Then
  \begin{equation*}
      \Exp\big[1(N > 0)\varphi\big(N^{-1}\big)\big]
      \le
      a\,\varphi\big((np)^{-1}\big).
  \end{equation*}
  In particular, for $\alpha\in (0, 1)$,
  \begin{equation*}
    \Exp\big[1(N > 0) N^{-\alpha}\big] \le 2(np)^{-\alpha}.
  \end{equation*}
\end{lemma}

\begin{proof}
  Employing the Jensen inequality, we observe
  \begin{align*}
    \Exp\big[1(N > 0)\varphi\big(N^{-1}\big)\big]
    &\le
    a\,\Exp\big[1(N > 0)\varphi\big((N+1)^{-1}\big)\big] \\
    &\le
    a\,\Exp\big[\varphi\big((N+1)^{-1}\big)\big] \\
    &\le
    a\,\varphi\big(\Exp\big[(N+1)^{-1}\big]\big).
  \end{align*}
  The expectation of $(N+1)^{-1}$ is known and equals \parencite[equation~3.4]{chao1972negative}
  \begin{equation*}
    \Exp\left[(N+1)^{-1}\right]
    =
    \frac{1 - (1-p)^{n+1}}{(n+1)p}
    \le
    \frac{1}{np}.
  \end{equation*}
  Since $\varphi$ is increasing, the main claim of the lemma follows.
  For the example, set $\varphi(x) = x^\alpha$, which is increasing and concave.
  It is easy to see that $k^{-\alpha} \le 2 (k+1)^{-\alpha}$ for all $k\in\NN$.
  Thus, we may set $a = 2$ and the claim follows.
\end{proof}

\section{Discussion}
\label{sec:discussion}

Our work introduces a general methodology for establishing upper bounds on the convergence rate of the empirical optimal transport cost to its population counterpart in various unbounded settings.
The approach is based on a generic decomposition strategy, which allows us to extend convergence results from compact to unbounded settings under suitable moment assumptions.
These moment assumptions are essentially sharp for various settings.
Our results apply to the one- and two-sample case. 

Despite the generally sharp character of our theory, certain refinements might be targeted in further investigations.
In the one-sample case, for example, it remains unclear under which conditions \emph{both} measures must admit finite $(s+\epsilon)$-th moments in order for the rate in \Cref{thm:unbounded-rates} to remain valid.
On the one hand, recent work by \textcite{del2022central} asserts looser moment requirements on $\nu$ when sampling from a discrete measure $\mu$ in the semi-discrete setting.
In a similar vein, the dual approach pursued in Appendix~\ref{app:dual_decomposition}, which is restricted to Lipschitz cost functions, also allows for weaker moment conditions on $\nu$ if only $\mu$ is sampled from.
On the other hand, \Cref{rem:sharp-moments} and the example in \Cref{app:non-integrable-potentials} show that imposing $p$-moment conditions on one of the measures alone is not sufficient to expect $p$-integrable dual solutions under rapidly growing cost functions.
Since the integrability of dual solutions is crucial for the proof of \Cref{thm:unbounded-rates}, this suggests that moment conditions on both $\mu$ and $\nu$ might indeed become necessary when treating general settings.
Another issue concerns the assumption of \emph{independence} between the samples in the two-sample setting.
While our main proof strategy would only allow for artificial forms of weak dependency, the alternative convergence result \Cref{thm:dual-unbounded-rates} in \Cref{app:dual_decomposition} does not require any dependency assumptions.
Again, however, this result requires Lipschitz costs, and the jury on whether assumptions that limit the degree of dependency are necessary in more general settings is still out.

\section*{Acknowledgements}
 T. Staudt gratefully acknowledges support of the DFG under Germany's Excellence Strategy - EXC 2067/1-
390729940. S. Hundrieser gratefully acknowledges support from the DFG RTG 2088. The authors thank Gilles Mordant and  Johan Segers for helpful discussions in the context of lower bounds. 

\printbibliography

\clearpage

\begin{appendix}

\section{Dual composition approach}
\label{app:dual_decomposition}

This appendix documents an alternative approach to convergence rates of the empirical optimal transport cost, based on the dual formulation \eqref{eq:dual} of optimal transport.
While the results allow for weaker dependency assumptions on the samples, it is essentially restricted to equicontinuous cost functions.
We first provide a general compositional convergence statement for supremum-type functionals in measurable spaces (\Cref{thm:dual-composition-rates}) and then use it for the analysis of the empirical optimal transport~cost (\Cref{thm:dual-unbounded-rates}).

Let $\X$ be a measurable space and $\mu,\mu'\in\P(\X)$ be two probability measures.
We consider a class $\F$ of real-valued measurable functions on $\X$ and study the associated functional
\begin{equation}\label{eq:sup-functional}
    \F(\mu, \mu')
    =
    \sup_{f\in\F}|\mint{(\mu-\mu')}f|,
\end{equation}
which is well-defined if $\mint\mu |f| + \mu' |f| < \infty$ for any $f\in \F$.
This operator is also called an \emph{integral probability metric} in the literature (see, e.g., \cite{muller1997integral}).
For technical reasons, we furthermore define the function class $\F_{x_0} = \{f - f(x_0)\,|\,f\in\F\}$ that is \emph{pinned} at $x_0\in\X$.
Since the objective in \eqref{eq:sup-functional} is shift-invariant, it holds that $\F(\mu, \mu') = \F_{x_0}(\mu, \mu')$ for valid arguments $\mu,\mu'\in\P(X)$.
Finally, for a subset $\U\subset\X$, we write
\begin{equation*}
  \|\F\|_\U
  =
  \sup_{f\in\F}\,\sup_{x\in\U}\, |f(x)|
\end{equation*}
to denote the largest uniform norm in $\F$ restricted to $\U$.
Note that pinning may decrease the norm of $\F$ (without affecting the functional), which is an effect that we exploit later.
The following result plays a similar role as \Cref{thm:composition-rates}, but is tailored to functionals as in~\eqref{eq:sup-functional}.

\begin{theorem}{composition rates, supremum-type functional}{dual-composition-rates}
  Let $\X$ be a measurable space, $\mu\in\P(\X)$, and $\F$ a class of real-valued measurable functions on $\X$.
  Assume $\mu = \sum_{l\in\NN} a_l \mu_l$, where $a\in\P(\NN)$ and $\mu_l\in\P(\X)$ for all $l\in\NN$.
  Further, assume there are $c_l > 0$ with $\|\F\|_{\supp\,\mu_l} \le c_l$, and $b_l \ge a_l$ such that $(b_l)_{l\in\NN}$ is monotonically decreasing and admits the existence of $l_n\in\NN$ with $b_{l_n} \asymp 1/n$.
  If there are $\alpha, \beta, \gamma \in (0, 1/2]$ and $r_l > 0$ with\footnotemark
  \begin{align*}
    &\Exp\,\F(\mu_l, \hat\mu_{l,n})
    \le
    r_l\,n^{-\alpha}
  \intertext{
  for all $l,n\in\NN$, $\sum_{l\in\NN} r_l b_l^{1-\beta} < \infty$, and $\sum_{l\in\NN} c_l b_l^{1-\gamma} < \infty$, then, for all $n\in\NN$,
}
    &\Exp\,\F(\mu, \hat\mu_n)
    \lesssim
    n^{-\min(\alpha, \beta, \gamma)}.
  \end{align*}
\end{theorem}

\footnotetext{Note that we do not necessarily require that $\F(\mu_l, \hat\mu_{l,n})$ and $\F(\mu, \hat\mu_n)$ are measurable, since the theorem remains valid if outer expectations are employed \parencite[Section~1.2]{Vaart1996}.}

\begin{proof}
Like in the proof of \Cref{thm:composition-rates}, we can assume that the samples $(X_i)_{i=1}^n$ are generated in a two-step procedure:
first, frequencies $N = (N_l)_{l\in\NN} \sim \Mult(n, a)$ are drawn, before $N_l$ observations $X_{l, 1}, \ldots, X_{l, N_l}$ are i.i.d.-sampled from $\mu_l$ for each $l\in\NN$.
We let $\hat\mu_{l,N_l} = \frac{1}{N_l} \sum_{i=1}^{N_l} \delta_{X_{l,i}}$ whenever $N_l > 0$.
Fix some $f\in\F$.
Noting that $\mint{\hat\mu_{l, N_l}}f \le c_l$, we find
\begin{align*}
    |\mint{(\mu - \hat\mu_n)}f|
    &=
    \left| \sum_{l\in\NN} a_l\mint{\mu_l} f - \frac{N_l}{n} \mint{\hat\mu_{l,N_l}}f\right| \\
    &\le
    \sum_{l\in\NN}a_l\, 1(N_l > 0)\, \big|\mint{\mu_l} f - \mint{\hat\mu_{l,N_l}}f\big|
    + \sum_{l\in\NN} c_l \left|a_l - \frac{N_l}{n}\right|,
\end{align*}
where we have used that $ab - cd = a(b - d) + (a-c)d$.
Taking the supremum over $f\in\F$ and the expectation over the observations yields
\begin{align*}
  \Exp\,\F(\mu, \hat\mu_n)
  &\le
  \sum_{l\in\NN}a_l\,\Exp\left[1(N_l > 0)\, \F(\mu_l, \hat\mu_{l, N_l})\right]
  + \frac{1}{n}\sum_{l\in\NN} c_l\,\Exp\,|N_l - na_l| \\
  &\le
  \sum_{l\in\NN}r_la_l\,\Exp\big[1(N_l > 0)\, N_l^{-\alpha}\big]
  + \frac{1}{n}\sum_{l\in\NN} c_l\,\Exp\,|N_l - na_l|,
\end{align*}
where we have conditioned on $N_l$ in the expectation and exploited that $\Exp\,\F(\mu_l, \hat\mu_{l,k}) \le r_l\, k^{-\alpha}$ for all $1 \le k \le n$ by assumption.
From here on, we can proceed like in the proof of \Cref{thm:composition-rates} after equation \eqref{eq:composite-rates-components} to establish the claim.
\end{proof}

The previous statement provides a way to  generalize convergence rates from bounded to unbounded settings by making use of the dual formulation \eqref{eq:dual} of optimal transport.
To this end, we have to formulate certain restrictions on the set of dual variables, joined by a bounded convergence assumption similar to $\BC(\kappa, \alpha)$.

\begin{assumption*}{$\DBC(\kappa, \alpha)$}{}
  Let $c$ be of the form \eqref{eq:marginal-cost-bound}, and let $\F$ and $\G$ be classes of real-valued Borel functions on $\X$ and $\Y$.
  Assume there exists $\Phi_c \subset \F\times\G$ that satisfies
  \begin{equation}\label{eq:dual-dbc-duality}
    \T_c(\mu, \nu)
    =
    \sup_{(f,g)\in\Phi_c} \mint\mu f + \mint\nu g
  \end{equation}
  for all $\mu\in\P(\X)$ and $\nu\in\P(\Y)$ with $\mint\mu \cX + \mint\nu \cY < \infty$.
  Further, assume there are $\kappa > 0$, $\alpha\in(0, 1/2]$, and pinning points $x_0\in\X$ and $y_0\in\Y$ such that the conditions
  \begin{subequations}
  \begin{alignat}{3}
    \Exp\,\F(\mu, \hat\mu_n)
    &\le
    \kappa\,r\,n^{-\alpha}
    \qquad&&\text{and}\qquad
    \Exp\,\G(\nu, \hat\nu_m)
    &&\le
    \kappa\,r\,m^{-\alpha},
    \label{eq:dual-bounded-convergence}%
    \\
    \|\F_{x_0}\|_{B_{\X}(r)}
    &\le
    \kappa\,r
    \qquad&&\text{and}\qquad
    \|\G_{y_0}\|_{B_{\Y}(r)}
    &&\le
    \kappa\,r
    \label{eq:dual-bound-pinned}%
  \end{alignat}
  \label{eq:dual-dbc-regularity}%
  \end{subequations}%
  hold for all $r \ge 1$, $\mu\in\P\big(B_\X(r)\big)$, $\nu\in\P\big(B_\Y(r)\big)$, and $n,m\in\NN$.
\end{assumption*}

This assumption imposes a considerable burden on the cost function, especially regarding the existence of suitable classes $\F$ and $\G$ of dual variables.
Generally, the function class $\Phi_c$ in \eqref{eq:dual-dbc-duality} can be restricted to pairs of $c$-concave functions \parencite[Theorem 5.10]{villani2008optimal}
\begin{equation*}
  \Phi_{c}
  =
  \big\{ (g^c,g^{cc}) \, |\,  g \colon \YC \to \RR\cup\{-\infty\}, g^{c} \not \equiv -\infty, g^{cc} \not \equiv -\infty \big\},
\end{equation*}
where $g^c(\cdot) = \inf_{y\in \YC} c(\cdot,y) - g(y)$ and $g^{cc}(\cdot) = \inf_{x\in \XC} c(x,\cdot) - g^c(x)$ denote the single and double $c$-transforms of $g\colon\Y \to \RR\cup\{-\infty\}$.
Unless additional regularity assumptions are placed upon the costs, however, elements of $\Phi_c$ may fail to be bounded on compact domains (see \Cref{app:non-integrable-potentials} for an example), preventing the validity of \eqref{eq:dual-bound-pinned}. 

An important example where Assumption~$\DBC(\kappa, \alpha)$ can be shown to hold are $L$-Lipschitz continuous costs on $\RR^d\times\RR^d$ for some $L > 0$.
Assuming $c(0,0) = 0$ for simplicity, the dominating marginal costs can be chosen as $\cX(x) = L\norm{x}$ and $\cY(x) = L\norm{y}$, while the classes $\F$ and $\G$ can be selected as the sets of all $L$-Lipschitz functions on $\RR^d$.
Based on \textcite[Theorem 3.3 and Lemma A.2]{hundrieser2022empirical}, condition \eqref{eq:dual-bounded-convergence} is then met for a rate $\alpha$ equal to $1/(2\vee d)$ (with an additional logarithmic term in case of $d = 2$).
By $L$-Lipschitz continuity of $\F$ and $\G$, condition \eqref{eq:dual-bound-pinned} is also fulfilled. 
Faster rates can, for example, be admitted if the cost function additionally fulfills Assumption~$\SMOOTH(1)$ from \Cref{subsec:SmoothCosts} in both components. Then, a similar rescaling argument as in the proof of \Cref{cor:smooth}, together with \textcite[Theorem 3.8 and Lemma A.3]{hundrieser2022empirical}, shows that condition \eqref{eq:dual-bounded-convergence} is satisfied for $\alpha = 2/(4 \vee d)$ (with an additional logarithmic term for $d = 4$). 

\begin{theorem}{unbounded rates, dual}{dual-unbounded-rates}
  Let $\X$ and $\Y$ be Polish, $\mu\in\P(\X)$, $\nu\in\P(\Y)$, and $c\colon\XY\to\RRplus$ lower semi-continuous such that $\DBC(\kappa, \alpha)$ is satisfied.
  Assume there are $1 < s \le 2$ and $\epsilon > 0$ such that the moments $\mint\mu \cX^{s+\epsilon}$ and $\mint\nu \cY^{s+\epsilon}$ are finite.
  Then, for all $n,m\in\NN$,
  \begin{equation*}%
    \Exp\,\big|T_c(\hat\mu_{n}, \hat\nu_{m}) - T_c(\mu, \nu)\big|
    \lesssim
    (n\wedge m)^{-\alpha} + (n\wedge m)^{-\frab{s-1}{s}}.
  \end{equation*}
\end{theorem}

We want to highlight a number of noteworthy observations about this result.
In some respects, \Cref{thm:dual-unbounded-rates} is more general than \Cref{thm:unbounded-rates}, while the latter allows for a greater flexibility in the choice of the cost function.
\begin{enumerate}[leftmargin=2em, label={\arabic*)}]
  \item The sharp separation of the observations $X = (X_i)_{i=1}^n$ and $Y = (Y_j)_{j=1}^m$ in equation \eqref{eq:dual-functional-bound} below implies that dependencies between $X$ and $Y$ can be arbitrary in \Cref{thm:dual-unbounded-rates}.
    In contrast, \Cref{thm:unbounded-rates} relies on independence between $X$ and $Y$, since joint conditioning on the number of observations $N_l$ and $M_l$, belonging to $\mu_l$ and $\nu_l$ respectively, is necessary.
  \item Equation \eqref{eq:dual-functional-bound} also has repercussions for the one-sample case, for which we find the inequality $|\T_c(\hat\mu_n, \nu) - \T_c(\mu, \nu)| \le \F(\mu, \hat\mu_n)$.
    Therefore, conditions that involve $\G$ explicitly may be dropped.
    In fact, only the requirement $\mint\nu|g| < \infty$ for all $g\in\G$ remains, and the explicit moment condition $\mint\nu\cY^{s+\epsilon}$ is not necessary anymore.
    In \Cref{thm:unbounded-rates}, on the other hand, the proof does not allow for the moment condition on $\nu$ to be dropped if only $\mu$ is sampled from.
  \item The mentioned benefits of \Cref{thm:dual-unbounded-rates} over \Cref{thm:unbounded-rates} (see points 1.\ and 2.\ above) critically rely on the fact that condition $\DBC(\kappa, \alpha)$ guarantees the existence of separate classes $\F$ and $\G$ of dual variables that are (a) well-behaved and (b) sufficiently rich for the dual representation of optimal transport to hold.
    This artificial separation of dual variables $f$ and $g$, which are intrinsically linked by the optimality criterion \parencite[Section~1.2]{santambrogio2015optimal}
    \begin{equation*}
      g(y) = \inf_{x\in\X} c(x, y) - f(x),
    \end{equation*}
    can only be realized in restricted settings, e.g., when the cost function is equicontinuous in its components.
    If this is not satisfied, classes $\F$ and $\G$ that are sufficiently rich for duality \eqref{eq:dual-dbc-duality} will typically fail to meet condition \eqref{eq:dual-dbc-regularity}.
\end{enumerate}

\begin{proof}[Proof of \Cref{thm:dual-unbounded-rates}]
We first note that the condition $\mint\mu\cX + \mint\nu\cY < \infty$ together with \eqref{eq:dual-bound-pinned} implies that $\mint\mu |f| + \mint\nu |g| < \infty$ for any $f\in\F$ and $g\in\G$.
We can follow the argumentation in \textcite[Theorem~2.2]{hundrieser2022empirical} to establish
\begin{equation}\label{eq:dual-functional-bound}
  |\T_c(\hat\mu_n, \hat\nu_m) - \T_c(\mu, \nu)|
  \le
  \F(\mu, \hat\mu_n) + \G(\nu, \hat\nu_m),
\end{equation}
which reduces our analysis to the objects $\F(\mu, \hat\mu_n)$ and $\G(\nu, \hat\nu_m)$, or, equivalently, $\F_{x_0}(\mu, \hat\mu_n)$ and $\G_{y_0}(\nu, \hat\nu_m)$.
From here on, we can essentially duplicate the proof of \Cref{thm:unbounded-rates} for both $\mu$ and $\nu$ separately, incorporating some adaptations that stem from exchanging Assumption~$\BC(\kappa, \alpha)$ with Assumption~$\DBC(\kappa, \alpha)$ and \Cref{thm:composition-rates} with \Cref{thm:dual-composition-rates}.
For illustration, this is done for $\mu$ in the following.

Analogously to the proof of \Cref{thm:unbounded-rates}, we may assume that $\cX \ge 2$.
Set $\X_l = \big\{x \in \X\,|\,2^l \le \cX(x) \le 2^{l+1}$\big\}, $a_l = \mu(\X_l)$, and $\mu_l = \mu|_{\X_l} / a_l$ (which is well-defined if $a_l>0$). 
Since the support of $\mu_l$ is contained in $B_{\X}(2^{l+1})$, condition \eqref{eq:dual-bound-pinned} implies $\|\F_{x_0}\|_{\supp\,\mu_l} \le c_l$ for the choice $c_l = \kappa\,2^{l+1}$.
Letting $r_l = \kappa\,2^{l+1}$ as well, condition \eqref{eq:dual-bounded-convergence} yields that
\begin{equation*}
  \Exp\,\F_{x_0}(\mu_l, \hat\mu_{l,n})
  \le
  r_l\,n^{-\alpha}
\end{equation*}
for all $l,n\in\NN$.
Similar to the proof of \Cref{thm:unbounded-rates}, we define $b_l = K\,2^{-l(s+\epsilon)}$ with $K = 2^{s+\epsilon} \mint\mu\cX^{s+\epsilon}$, which guarantees $a_l \le b_l$.
Letting $\beta = \gamma = (s-1) / s$, finiteness of the sums $\sum_{l\in\NN} r_l \smash{b_l^{1-\beta}}$ as well as $\sum_{l\in\NN} c_l\smash{b_l^{1-\gamma}}$ follows as in \eqref{eq:bounded-rates-constants}.
Therefore, all conditions for application of \Cref{thm:dual-composition-rates} to $\F_{x_0}$ are satisfied, and we conclude that $$\Exp\,\F(\mu, \hat\mu_n) \lesssim n^{-\min(\alpha, (s-1)/2)}.$$
Since analogous arguments are also valid for $\G(\nu, \hat\nu_m)$, the claim follows.
\end{proof}

\section{Auxiliary results and omitted proofs}
\label{app:auxiliaryProofs}

This appendix contains \Cref{lem:scaled-function}, which supports the cost scaling arguments used in \Cref{subsec:LocallyLipschitz} and \ref{subsec:SmoothCosts}, as well as the proofs for inequalities \eqref{eq:ConvergenceRateWassersteinNull} and \eqref{eq:ConvergenceRateWassersteinAlternative}.

\begin{lemma}{}{scaled-function}
  Let $f\colon \RR^d \!\to \RR$ be locally Lipschitz.
  For $p, \beta > 0$ and $r \ge 1$, set $f_{r,p,\beta} \colon \B \to \RR$,
  \begin{equation}\label{eq:scaled-function}
    f_{r,p,\beta}(u)
    =
    \frac{f\big(r^{1/p}u^{\beta/p}\big)}{r},
  \end{equation}
  where $\B$ denotes the Euclidean unit ball and where $u^{\beta/p}$ is shorthand for $\|u\|^{\beta/p - 1} u$.
  Let $B_\X(r) = \big\{x\in\RR^d\,|\,\|x\|^p \le r\big\}$. The following three statements hold.
  \begin{enumerate}[topsep=0.5em, itemsep=0.5em]
    \item If $f \le r$ on $B_\X(r)$, then $f_{r,p,\beta} \le 1$ on $\B$.
    \item If the first partial derivatives of $f$ at all $x\in B_\X(r)$ where $f$ is differentiable are bounded by $\|x\|^{p-1} + r^{1-1/p}$, then $f_{r,p,\beta}$ is $2d\beta/p$-Lipschitz on $\B$ for $\beta > \max(1, p)$.
    \item If all partial derivatives of $f$ of order $k\in\{1, 2\}$ exist at all $x\in B_\X(r)$ and are bounded by $\|x\|^{p-k} + r^{1-k/p}$, then the partial derivatives of $f_{r,p,\beta}$ exist and are bounded by $5d\beta^2/p^2$ for $\beta > \max(2, 2p)$.
   \end{enumerate}
\end{lemma}

\begin{proof}
  The first statement trivially follows from definition \eqref{eq:scaled-function}.
  For the second statement, first observe that $f_{r, p, \beta}$ is the composition of the functions $\frac{1}{r}f|_{B_{\X}(r)}$ and $u \mapsto r^{1/p}u^{\beta/p}$.
  Both functions are Lipschitz if $\beta > p$, implying that $f_{r, p, \beta}$ is Lipschitz as well.
  In order to bound the Lipschitz constant, we have to control the gradient.
  The $i$-th partial derivative of $f_{r, p, \beta}$ is given~by
  \begin{align*}
    \big|\partial_i f_{r, p, \beta}(u)\big|
    &=
    \big|r^{1/p - 1} \nabla f\big(r^{1/p}u^{\beta/p}\big) \cdot \partial_i u^{\beta/p}\big| \\
    &\le
    r^{1/p - 1} \big\|\nabla f\big(r^{1/p}u^{\beta/p}\big)\big\|\, \big\|\partial_i u^{\beta/p}\big\|
  \end{align*}
  whenever $f$ is differentiable at $r^{1/p} u^{\beta/p}$.
  Here, $\partial_i u^{\beta/p}$ denotes component-wise differentiation of the function $u \mapsto u^{\beta/p} = \|u\|^{\beta/p -1}u$ and $\cdot$ is the dot product.
  Exploiting the assumption of bounded partial derivatives of $f$, a brief calculation shows
  \begin{equation}\label{eq:scaled-gradient-bound}
    \big\|\nabla f\big(r^{1/p}u^{\beta/p}\big)\big\|
    \le
    \sqrt{d} \, r^{1-1/p} \left(\|u\|^{\beta - \beta/p} + 1\right),
  \end{equation}
  where the factor $\sqrt{d}$ stems from the fact that the norm of a vector in $\RR^d$ is bounded by $\sqrt{d}$ times its largest entry.
  We next observe
  \begin{align}
    \big\|\partial_i u^{\beta/p}\big\|
    &=
    \big\|\big(\partial_i \|u\|^{\beta/p - 1}\big) u + \|u\|^{\beta/p - 1} \partial_i u \big\| \nonumber\\
    &\le
    \big|\partial_i \|u\|^{\beta/p - 1}\big| \|u\| + \|u\|^{\beta/p-1} \|\partial_i u\| \nonumber\\
    &\le
    (\beta/p - 1) \|u\|^{\beta/p - 3} |u_i| \|u\| + \|u\|^{\beta/p-1} \nonumber\\
    &\le
    \frac{\beta}{p} \|u\|^{\beta/p-1},
    \label{eq:scaled-partial-first-bound}
  \end{align}
  where we have exploited $\beta/p > 1$, $\|\partial_i u\| = 1$, and $|u_i| \le \|u\|$.
  Stitching the previous inequalities together yields
  \begin{align*}
    \big|\partial_i f_{r, p, \beta}(u)\big|
    \le
    \frac{\sqrt{d} \beta}{p} \left(\|u\|^{\beta-\beta/p} + 1\right) \|u\|^{\beta/p-1} \\
    =
    \frac{\sqrt{d} \beta}{p} \left(\|u\|^{\beta-1} + \|u\|^{\beta/p-1}\right).
  \end{align*}
  Therefore, since $\|u\| \le 1$, all partial derivatives of $f_{r,p,\beta}$ are bounded by $2\sqrt{d}\beta/p$ if $\beta > \max(1, p)$.
  In particular, the gradient of $f_{r,p,\beta}$ is in norm bounded by $2d\beta/p$.
  We conclude that $f_{r,p,\beta}$ is $2d\beta/p$-Lipschitz.

  For the third statement, assume that $f$ is twice differentiable.
  If $\beta > 2p$, we note that $f_{r, p, \beta}$ is twice differentiable as well.
  The first derivatives have been calculated and bounded above.
  For the second derivatives, let $Hf(x)$ denote the Hessian matrix of $f$ at $x\in B_\X(r)$ and let $\|Hf(x)\|$ be its operator norm.
  Then
  \begin{align*}
    \big|\partial_i\partial_j f_{r, p, \beta}(u)\big|
    &=
    \left| r^{2/p -1}\,\partial_i u^{\beta/p} \cdot Hf\big(r^{1/p}u^{\beta/p}\big)\, \partial_j u^{\beta/p} + r^{1/p - 1} \nabla f\big(r^{1/p} u^{\beta/p}\big) \cdot \partial_i\partial_j u^{\beta/p} \right| \\
    &\le
    r^{2/p -1}\big\|\partial_i u^{\beta/p}\big\|\cdot \big\|\partial_j u^{\beta/p}\big\|\cdot \big\|Hf\big(r^{1/p}u^{\beta/p}\big)\big\| + r^{1/p - 1} \big\|\nabla f\big(r^{1/p} u^{\beta/p}\big)\big\| \cdot \big\|\partial_i\partial_j u^{\beta/p} \big\|.
  \end{align*}
  The only terms on the right-hand side that have not been bounded already are the Hessian of $f$ and the second derivatives of $u^{\beta/p}$.
  Since $\|Hf\|$ is bounded by $d$ times the largest entry of $Hf$, the assumption of bounded second partial derivatives yields, similarly to \eqref{eq:scaled-gradient-bound},
  \begin{equation}\label{eq:scaled-hessian-bound}
    \big\|Hf\big(r^{1/p}u^{\beta/p}\big)\big\|
    \le
    d\,r^{1 - 2/p} \left(\|u\|^{\beta - 2\beta/p} + 1\right).
  \end{equation}
  For the second partial derivatives of $u^{\beta/p}$, we note that $\partial_i\partial_j u = 0$ and recall
  the calculations in \eqref{eq:scaled-partial-first-bound}. We find
  \begin{align}
    \big\|\partial_i\partial_j u^{\beta/p}\big\|
    &=
    \left\| \left(\partial_i\partial_j\|u\|^{\beta/p - 1}\right) u + \left(\partial_i\|u\|^{\beta/p - 1}\right)\partial_j u + \left(\partial_j\|u\|^{\beta/p - 1}\right)\partial_i u + \|u\|^{\beta/p-1} \partial_i\partial_j u \right\| \nonumber \\
    &\le
    \big|\partial_i\partial_j\|u\|^{\beta/p - 1}\big| \|u\|
    +
    \big|\partial_i \|u\|^{\beta/p-1}\big|
    +
    \big|\partial_j \|u\|^{\beta/p-1}\big| \nonumber \\
    &\le
    (\beta/p - 1)\,|\beta/p - 3|\, \|u\|^{\beta/p-4} |u_i||u_j| + 1_{(i=j)} (\beta/p) \|u\|^{\beta/p-2}
    +
    2(\beta/p - 1) \|u\|^{\beta/p - 2} \nonumber \\
    &\le
    \left(\frac{2\beta}{p}\right)^2 \|u\|^{\beta/p - 2},
    \label{eq:scaled-partial-second-bound}
  \end{align}
  where we have used $\beta / p > 2$ in order to bound the constants.
  Putting inequalities \eqref{eq:scaled-gradient-bound} to \eqref{eq:scaled-partial-second-bound} together, we obtain
  \begin{equation*}
    \big|\partial_i\partial_j f_{r, p, \beta}(u)\big|
    \le
    d \left(\frac{\beta}{p}\right)^2 \left(\|u\|^{\beta-2} + \|u\|^{2\beta/p-2}\right)
    +
    \sqrt{d}\left(\frac{2\beta}{p}\right)^2 \left(\|u\|^{\beta-2} + \|u\|^{\beta/p -2}\right).
  \end{equation*}
  The condition $\beta/p > \max(2, 2p)$ together with $\|u\|\le 1$ thus implies that the second partial derivatives of $f_{r,p,\beta}$ are bounded by $5 d \beta^2/p^2$.
\end{proof}

\begin{proof}[Proof of Inequality \eqref{eq:ConvergenceRateWassersteinNull}]
Since $T_{d^p}^{1/p}$ for $p \ge 1$ defines a metric on the space of probability measures $\PC_p(\YC)$ with finite $p$-th moment (see Chapter 6 of \cite{villani2008optimal}) it follows for any $\mu, \nu, \tau \in\PC_p(\YC)$ via the triangle inequality that
\begin{equation*}
  T_{d^p}(\mu, \nu)
  =
  \left( T_{d^p}^{1/p}(\mu, \nu) \right)^p
  \leq
  \left(T_{d^p}^{1/p}(\mu, \tau) + T_{d^p}^{1/p}(\tau, \nu)\right)^p
  \leq
  2^{p} \big(T_{d^p}(\mu, \tau) + T_{d^p}(\tau, \nu)\big).
\end{equation*}
Hence, inequality \eqref{eq:ConvergenceRateWassersteinNull} immediately follows once we show, for any $r\geq 1$ and $n\in \NN$, that
\begin{equation}\label{eq:ConvergenceWassersteinNull_OneSample}
  \sup_{\mu \in \PC(B_\X(r))}\Exp\,\big| \T_{d^p}(\hat\mu_n, \mu) \big|
  \lesssim
  r \varphi_{p,t}(n).
\end{equation}
Our proof for \eqref{eq:ConvergenceWassersteinNull_OneSample} relies on quantitative stability bounds by \textcite{weed2019sharp}. 
To apply them, we consider the rescaled metric $d'\coloneqq d/\big(4 k_\XC^{1/t} r^{1/p}\big)$ with $t$ and $k_\XC$ as in $\DIM(t)$.
Then the set $B_\X(r)$ has diameter at most one with respect to $d'$, and for any $\mu, \nu\in \PC\big(B_\X(r)\big)$ we find
\begin{equation*}
  T_{d^p}(\mu, \nu)
  =
  4^p r k_\XC^{p/t}T_{d'^p}(\mu, \nu).
\end{equation*}
A straightforward computation shows that $\covering(\delta, U, d') = \covering\big(4  k_\XC^{1/t} r^{1/p}\delta, U, d\big)$ and by $\DIM(t)$ we find for $U \subset B_\X(r)$ that 
\begin{align*}
  \covering(\delta, U, d')
  =
  \covering\big(4 k_\XC^{1/t} r^{1/p} \delta, U, d\big)
  \leq
  1+k_\XC \left(\frac{\diam(U)}{4 k_\XC^{1/t} r^{1/p} \delta}\right)^t
  =
  1+\left(\frac{\diam(U)}{4 r^{1/p} \delta}\right)^t \leq 1+\delta^{-t}
\end{align*}
Using Propositions 1, 3 and 4 of \textcite{weed2019sharp}, it follows for any  $\mu \in \PC\big(B_\X(r)\big)$ and $k^*\in \NN$ (by choosing  $S\coloneqq \supp(\mu)$ for each $k\in \{1, \dots, k^*\}$) that 
\begin{align*}
  \EE\,\T_{d'^p}(\hat\mu_n, \mu)
  &\leq
  3^{-k^*p} + \sum_{k = 1}^{k^* } 3^{-(k-1)p} \sqrt{\frac{\covering\big(3^{-(k+1)}, \supp\,\mu, d'\big)}{n}} \\
  &\leq
  3^{-k^* p} + \frac{3^{-2p}}{\sqrt{n}}\sum_{k = 1}^{k^*} \left[3^{-(k+1)(p -t/2)} + 3^{-(k+1)p}\right]
\end{align*}
If $t/2<p$, let $k^*\rightarrow \infty$ and note that
\begin{align*}
  \EE\,\T_{d'^p}(\hat\mu_n, \mu)
  \leq
  0 + \frac{3^{-2p}}{\sqrt{n}}\left( \frac{3^{2(p-t/2)}}{1-3^{p-t/2}} + \frac{3^{-2p}}{1-3^{-p}}\right)
  \asymp
  n^{-1/2} \asymp\varphi_{p,t}(n).
\end{align*}
If $t/2 = p$, choose $k^*= k^*(n)=\lceil\frac{1}{2}\log_3(n)/p\rceil$, which results in 
\begin{align*}
  \EE\,\T_{d'^p}(\hat\mu_n, \mu)
  \leq
  \frac{1}{\sqrt{n}} + \frac{3^{-2p}}{\sqrt{n}}\left( \frac{\log_3(n)}{2p}+1+ \frac{3^{-2p}}{1-3^{-p}} \right)  \asymp  n^{-1/2}\log(n+1) \asymp \varphi_{p,t}(n).
\end{align*}
Finally, if $t/2 >p$, consider $k^* = k^*(n)= \lceil\frac{1}{t}\log_3(n)\rceil$, which yields
\begin{align*}
  \EE[\T_{d'^p}(\hat\mu_n, \mu)] \leq n^{-p/t} + \frac{3^{-2p}}{\sqrt{n}}\left( \frac{3^{(k^*+1)(t/2-p)} - 3^{t-2p}}{3^{t/2-p} - 1} + \frac{3^{-2p}}{1-3^{-p}} \right).
  \end{align*}
A straightforward computation then yields that $ 3^{(k^*+1)(t/2-p)} \leq 3^{t-2p} n^{1/2 - p/t}$, which asserts,
\begin{align*}
   \EE[\T_{d'^p}(\hat\mu_n, \mu)] \leq n^{-p/t} + \frac{3^{-2p}}{\sqrt{n}}\left( \frac{3^{t-2p} n^{1/2 - p/t} - 3^{t-2p}}{3^{t/2-p} - 1} + \frac{3^{-2p}}{1-3^{-p}} \right)\asymp n^{-p/t} \asymp \varphi_{p,t}(n).
\end{align*}
From this we conclude the validity of \eqref{eq:ConvergenceWassersteinNull_OneSample}, which finishes the proof of Inequality \eqref{eq:ConvergenceRateWassersteinNull}.
 \end{proof}

\begin{proof}[Proof of Inequality \eqref{eq:ConvergenceRateWassersteinAlternative}]

First note by \textcite[Theorem 5.10(iii)]{villani2008optimal} that for any $\mu\in \PC\big(B_\X(r)\big)$ and $\nu \in \PC\big(B_\Y(r)\big)$, 
\begin{equation}\label{eq:KR}
  T_{d}(\mu, \nu)
  =
  \sup_{f\in \FC(r)} \mu f + \nu f^d
\end{equation}
where $\FC(r) \coloneqq \{ f\colon B_\X(r) \rightarrow \RR \;|\; \exists\, g \colon \YC \rightarrow \RR, f = \inf_{y\in \YC} d(\cdot,y) - g(y)\}$ and $f^d \colon B_\Y(r) \rightarrow \RR$, $y\mapsto \inf_{x\in B_\X(r)} f(x) - d(x,y)$ denotes the $d$-transform of $f$.
Further, since the right-hand side of \eqref{eq:KR} is invariant  under constant shifts of $f$, we can replace $\FC(r)$ in \eqref{eq:KR} by the pinned function class $\FC_{x_0}(r) = \{ f - f(x_0) \;|\; f\in \FC(r)\}$.
As a consequence, we conclude for empirical measures $\hat \mu_n$ and $\hat \nu_m$ that 
\begin{align*}
  \EE \left|  T_{d}(\hat\mu_n, \hat\nu_m) -  T_{d}(\mu, \nu) \right|
  \leq
  \EE\left[\sup_{f\in \FC_{x_0}(r)} \Big|(\hat \mu_n - \mu)f\Big| \right]+ \EE\left[\sup_{f\in \FC_{x_0}(r)} \left|(\hat \nu_m - \nu)f^d\right|\right].
\end{align*}
To examine the expectations on the right-hand side, we use methods of empirical process theory, relying on bounds on the metric entropy of $\FC_{x_0}(r)$ and $\smash{\big(\FC_{x_0}(r)\big)^d}$.  
In this context, note that since $\FC_{x_0}(r)$ consists of functions which are the infimum over $1$-Lipschitz functions, these functions are as well.
Further, $\FC_{x_0}(r)$ admits envelope function $d(\cdot,x_0)$ and is therefore uniformly bounded by $r$.
Hence, $\FC_{x_0}(r)$ is contained in $ \Lip_{1,r}(B_\X(r))$ the class of $1$-Lipschitz functions on $B_\X(r)$ which are bounded by $r$.
Uniform metric entropy bounds for Lipschitz functions on metric space are available by \textcite[Section 9]{Kolmogorov1961} and assert for any $\delta > 0$ that
\begin{align*}
  \log\covering\big(\delta, \Lip_{1,r}(B_\X(r)), \|\cdot\|_\infty\big)
  & \lesssim 
  \begin{cases}
    \log\left(2 \lceil\frac{2r}{\delta}\rceil +1\right)\covering\left(\frac{\delta}{4}, B_\X(r), d\right) \!\!\!\!& \text{ if } B_\X(r) \text{ is disconnected,} \\
    \log\left(2 \left\lceil \frac{2r}{\delta}\right\rceil+1\right) +\covering\left(\frac{\delta}{4}, B_\X(r), d\right) \!\!\!\!& \text{ if } B_\X(r) \text{ is connected,}
  \end{cases}
\end{align*}
for a universal hidden constant.
In conjunction with Assumption $\DIM(t)$, we obtain by following along the proof of Theorem 2.1 in \textcite{hundrieser2022empirical}, in case $B_\X(r)$ is disconnected, that
\begin{align*}
  \EE\bigg[ \sup_{f \in \FC_{x_0}(r)}  (\hat \mu_n - \mu)(f)\bigg]
  \lesssim&
  \inf_{\epsilon\in (0, r]} \epsilon + \frac{1}{\sqrt{n}} \int_{\epsilon/4}^r \sqrt{\log\covering(\delta,\Lip_{1,r}(\XC), \|\cdot\|_\infty)}\dif \delta \\
  \leq&
  \inf_{\epsilon\in (0, r]} \epsilon + \frac{1}{\sqrt{n}} \int_{\epsilon/4}^r \sqrt{\log\left(2\left\lceil \frac{2r}{\delta}\right\rceil+1\right)\left( 1 + k_\XC \left(\frac{16 r}{\delta}\right)^{t}\right)} \dif \delta\\
  \;\;=&
  \;r \inf_{\epsilon'\in (0, 1]} \epsilon' + \frac{1}{\sqrt{n}} \int_{\epsilon'/4}^1  \sqrt{\log\left(2\left\lceil \frac{2}{\delta'}\right\rceil+1\right)\left( 1 + k_\XC \left(\frac{16}{\delta'}\right)^{t}\right)} \dif \delta',
\end{align*}
where we used in the last equality the substitutions $\delta' = \delta/r$ and $\epsilon' = \epsilon/r$.
Overall, this asserts
\begin{align*}
  \EE\bigg[ \sup_{f \in \FC_{x_0}(r)}  (\hat \mu_n - \mu)(f)\bigg]
  \lesssim
  r \begin{cases}
 	n^{-1/2} & \text{ if } t < 2 \text{ for }\epsilon' = 0,\\
 	n^{-1/2} \log(n)^2 & \text{ if } t = 2 \text{ for } \epsilon' = 4 n^{-1/2},\\
 	n^{-1/t} \log(n)&\text{ if } t > 2 \text{ for } \epsilon' = 4 n^{-1/t}.
 \end{cases}
\end{align*}

Likewise, if $B_\X(r)$ is connected, it follows using concavity of $x\mapsto x^{1/2}$ on $\RRplus$ that 
\begin{align*}
  \EE\bigg[ \sup_{f \in \FC_{x_0}(r)}  (\hat \mu_n - \mu)(f)\bigg]
  \lesssim&
  \;r \inf_{\epsilon'\in (0, 1]} \epsilon' + \frac{1}{\sqrt{n}} \int_{\epsilon'/4}^1 \log\left(2\left\lceil \frac{2}{\delta'}\right\rceil+1\right) + 1 + k^{1/2}_\XC \left(\frac{16}{\delta'}\right)^{t/2} \dif \delta'\\
  \lesssim&
  \;
  r \begin{cases}
    n^{-1/2} & \text{ if } t< 2 \text{ for }\epsilon' = 0,\\
    n^{-1/2} \log(n) & \text{ if } t = 2 \text{ for } \epsilon = 4 n^{-1/2},\\
    n^{-1/t}&\text{ if } t > 2 \text{ for } \epsilon = 4 n^{-1/t}.
  \end{cases}	  
\end{align*}
Moreover, since $d$-transformation is a Lipschitz operation, it does not increase the uniform metric entropy that \parencite[Lemma 2.1]{hundrieser2022empirical}, i.e, for any $\delta > 0$ it holds that
\begin{equation*}
  \log\covering\big(\delta,(\FC_{x_0}(r))^d, \|\cdot\|_\infty\big)
  \leq
  \log\covering\big(\delta,\FC_{x_0}(r), \|\cdot\|_\infty\big). 
\end{equation*}
This yields an identical upper bound for the remaining term $\EE\big[\sup_{f \in \FC_{x_0}(r)}  (\hat \nu_m - \nu)(f^d)\big]$ as for $\smash{\EE\big[ \sup_{f \in \FC_{x_0}(r)}  (\hat \mu_n - \mu)(f)\big]}$, and asserts the validity of \eqref{eq:ConvergenceRateWassersteinAlternative}.
\end{proof}

\section{Non-integrability of dual solutions}
\label{app:non-integrable-potentials}

In this appendix, we establish a complementary example to \Cref{lem:lp-potentials}, showcasing that dual solutions cannot be expected to be $L^p$-integrable if one of the marginal probability measures does not admit enough moments.
We construct the example for $\X = \Y = \RR_+$ under the cost function $c\colon \RR_+\times\RR_+\to\RR_+$, $(x,y) \mapsto (x - y)^{\gamma}$ for $\gamma\in 2\,\NN$, which is dominated by the marginal costs $\cX(x) = 2^\gamma |x|^\gamma$ and $\cY(y) = 2^\gamma |y|^\gamma$.  
For any fixed $p > 1$, we will propose measures $\mu,\nu\in\P(\RR_+)$ and values $\gamma$ such that
\begin{equation*}
  \mint(\mu\otimes\nu) (\cX \oplus \cY) < \infty,
  \quad
  \mint\mu\cX^p < \infty,
  \quad
  \mint\nu\cY^p = \infty,
  \quad
  \mint\mu |f|^p = \infty,
  \quad
  \mint\nu |g|^p = \infty,
\end{equation*}
for all dual solutions $f$ and $g$ that satisfy $\T_c(\mu, \nu) = \mint\mu f + \mint\nu g$.
This in particular establishes that $p$-moment conditions on only one measure are not generally sufficient to ensure $p$-integrability of any of the dual solutions.

Let $\alpha, \beta > 0$. We define $\mu$ and $\nu$ via their cumulative distribution functions 
\begin{equation*}
  F_\mu(t)
  =
  1-\big(1-(1 \wedge t)\big)^{\alpha}
  \quad\text{and}\quad
  F_\nu(t)
  =
  1- (1\vee t)^{-\beta}
  \quad\text{for}~
  t \geq 0.
\end{equation*}
Then, $\mu$ is compactly supported on $[0,1]$, asserting $\mint\mu \cX^q < \infty$ for all $q > 0$.
If $\beta > \gamma$, it also holds that $\mint\nu \cY < \infty$, implying $\mint{(\mu\otimes \nu)}(\cX \oplus \cY) <\infty$.
Hence, by Theorem 5.10iii of \textcite{villani2008optimal}, we conclude the existence of dual solutions $(f,g)\in L^1(\mu)\times L^1(\nu)$ such that $f\oplus g \leq c$ as well as $T_c(\mu, \nu) = \mint\mu f + \mint\nu g < \infty$.
Furthermore, by Theorem 2.10 of \textcite{bobkov2019one}, the optimal transport coupling $\pi$ between $\mu$ and $\nu$ is determined via the transport map
\begin{equation*}
  T\colon (0,1) \mapsto (1,\infty), \quad t \mapsto F_\nu^{-1} \circ F_\mu(t) = (1-t)^{-\alpha/\beta}.
\end{equation*}
Invoking Theorem 10.28 of \textcite{villani2008optimal}, which is applicable by Example 10.35 in that reference, the optimal transport map $T$ and dual solution $f$ are linked $\mu$-almost surely via the relation 
\begin{equation*}
  \textstyle f'(t)
  =
  -\nabla_x c\big(t,T(t)\big)
  =
  -\gamma \left(t-(1-t)^{-\alpha/\beta}\right)^{\gamma-1}.
\end{equation*}
It thus follows that $f$ is uniquely determined on $[0, 1]$ up to a constant shift. Since $f'(t)$ essentially grows like $(1-t)^{-(\gamma-1)\alpha/\beta}$ for $t$ close to $1$, elementary computations show that there is an $\epsilon>0$ such that
\begin{equation*}
  |f(t)| \asymp(1-t)^{-(\gamma-1)\alpha/\beta + 1}
  \quad\text{for all}~
  t\in (1-\epsilon,1),
\end{equation*}
provided $(\gamma-1)\alpha \neq \beta$.
By definition of $\mu$, it follows that
\begin{equation*}
  \mint\mu|f|^p
  \gtrsim
  \int_{1-\epsilon}^1 (1-t)^{-p(\gamma-1)\alpha/\beta+p} (1-t)^{\alpha-1} \dif t, 
\end{equation*}
the right-hand side of which is equal to $\infty$ if and only if
\begin{equation*}
  p + \alpha \leq p(\gamma-1)\frac{\alpha}{\beta}.
\end{equation*}
This condition, together with $\beta > \gamma\in 2\,\NN$ and $(\gamma-1)\alpha > \beta$, is for example satisfied for
\begin{equation*}
  \alpha = \beta/2
  \quad\text{and}\quad
  \beta = \gamma + 1
  \quad\text{if}\quad
  \gamma \ge \frac{3p+1}{p-1}.
\end{equation*}
In fact, this choice of parameters also satisfies $p\gamma \ge \beta$, which implies
\begin{equation*}
  \mint\nu \cY^p
  =
  \int_1^\infty 2^{\gamma} t^{p\gamma}\,\beta t^{-\beta-1}\dif t
  =
  2^{\gamma} \beta \int_1^\infty t^{p\gamma - \beta -1}\dif t
  = \infty.
\end{equation*}
Since $c(x,y) = |x-y|^\gamma \geq |y|^\gamma$ for any $x\in [0,1]$ and $y\in [1,\infty)$, we can additionally conclude $\mint{(\mu \otimes \nu )}c^p \gtrsim \mint\nu \cY^p = \infty$. In order to show that $\mint\nu|g|^p = \infty$, we observe, similar as for $f$,
\begin{equation*}
  g'(s)
  =
  -\nabla_x c\big(s, T^{-1}(s)\big)
  =
  -\gamma \big(s - T^{-1}(s)\big)^{\gamma-1}.
\end{equation*}
Since $T^{-1}(s)\in (0,1)$ for any $s\in (1,\infty)$, the derivative $g'$ behaves like $s^{\gamma-1}$ for large $s$.
Thus, we find $|g(s)| \asymp s^\gamma$ and $\mint\nu |g|^p = \infty$ as well.

\end{appendix}

\end{document}